\documentclass[12pt]{amsart}
\usepackage[margin=2cm]{geometry}
\calclayout

\usepackage{amsmath,amsthm,amssymb,amsfonts}
\usepackage{bm}
\usepackage{mathrsfs}
\usepackage{esint}

\usepackage[sort,square,numbers]{natbib}

\usepackage{colonequals}
\usepackage{enumitem}
\usepackage{hyperref}

\usepackage{color,xcolor}

\newtheorem{theorem}{Theorem}[section]

\newtheorem{proposition}{Proposition}[section]
\newtheorem{lemma}{Lemma}[section]

\newtheorem{remark}{Remark}[section]

\numberwithin{equation}{section}

\newcommand{\de}{\delta}
\newcommand{\la}{\lambda}
\newcommand{\ep}{\varepsilon}
\newcommand{\para}{\alpha_s^\frac12 h^\frac{2-\beta}2}

\newcommand{\expo}{h^\frac{6-\beta}4}
\newcommand{\expos}{h^\frac43}

\newcommand{\N}{\mathbb N}
\newcommand{\R}{\mathbb R}
\newcommand{\Z}{\mathbb Z}

\renewcommand{\u}{\bm{\mathrm{u}}}
\newcommand{\w}{\bm{\mathrm{w}}}

\newcommand{\rel}{\mathrm{rel}}

\newcommand{\factor}{h^\frac{2+\beta}4}
\newcommand{\factors}{h^\frac{2+\beta}2}
\newcommand{\mask}{m\left[h^\de k-\frac{\alpha_s^\frac14 r \ell }\factor\right]}
\newcommand{\maskk}{m^2\left[h^\de k-\frac{\alpha_s^\frac14 r \ell }\factor\right]}
\newcommand{\maskj}{m\left[h^\de j-\frac{\alpha_s^\frac14 r \ell }\factor\right]}
\newcommand{\dr}{\,\mathrm{d}r}
\newcommand{\dx}{\,\mathrm{d}x}
\newcommand{\dt}{\,\mathrm{d}\theta}

\usepackage[foot]{amsaddr}

\title[Curvature-induced wrinkling]{Wrinkling of an elastic sheet floating on a liquid sphere}
\author[Peter Bella]{Peter Bella}
\email[Peter Bella]{peter.bella@udo.edu}
\address[Peter Bella]{TU Dortmund\\
Fakultät für Mathematik\\
Vogelpothsweg 87\\
44227 Dortmund\\
Germany
}
\author[Carlos Rom\'{a}n]{Carlos Rom\'{a}n}
\email[Carlos Rom\'{a}n]{carlos.roman@uc.cl}
\address[Carlos Rom\'{a}n]{Facultad de Matem\'aticas e Instituto de Ingenier\'ia Matem\'atica y Computacional, Pontificia Universidad Cat\'olica de Chile, Vicu\~na Mackenna 4860, 7820436 Macul, Santiago, Chile}

\thanks{Carlos Rom\'{a}n acknowledges funding from ANID FONDECYT Regular 1231593, and both authors acknowledge support 
from the German Research Foundation DFG in context of the Emmy Noether Junior Research Group BE 5922/1-1.}
\thanks{Final version. Accepted for publication in Calculus of Variations and Partial Differential Equations.}
\date{\today}
\begin{document}

\begin{abstract}
A thin circular elastic sheet floating on a drop-like liquid substrate is deformed due to incompatibility between the curved substrate and the planar sheet. We adopt a variational viewpoint by minimizing the non-convex membrane energy together with a higher-order convex bending energy. Focusing on thin sheets, we expand the minimum of the energy in terms of a small thickness ratio $h$, and identify the first two terms of this expansion. The leading‑order term arises from the minimization of a family of one‑dimensional relaxed problems, while for the next‑order term we establish lower and upper bounds.
This generalizes the previous work [Bella \& Kohn, Philos. Trans. Roy. Soc. A 2017] to the physically relevant case of a liquid substrate. 
\end{abstract}

\dedicatory{Dedicated to the memory of Robert V. Kohn}
\maketitle

\section{Introduction}
\subsection{The model}
We consider a circular elastic sheet floating on a liquid spherical substrate. Assuming the center of the sheet is the north pole, the sheet tries to wrap the spherical cap, which is not possible without inducing compressive stresses due to the positive curvature of the sphere (substrate). Since the sheet is thin, these stresses are relaxed not by compression but by small-scale wrinkling out of plane. In physics, this type of problem is very popular as a prototypical example of instabilities, and often serves as a testing ground for newly developed theories~\cite{KuRuDaMe20}. Mathematically, it belongs to a large class of problems where the microstructure can be explained via minimization of non-convex energies---sometimes referred to as energy-driven pattern formation~\cite{Ko07}.

The model of a circular sheet floating on a deformable ball was, from the mathematical point of view, previously studied by Kohn and the first author~\cite{BeKo17}, under the assumption that the underlying compliant substrate is elastic. While mathematically convenient---since a special balance among several terms in the energy functional ensures that the limit of the leading‑order contributions is independent of the thickness (see Sections~\ref{sec:heuristic} and \ref{sec:effectivefunctional} for more details)---physical experiments favor the setting with a liquid substrate~\cite{PaDeSa+15, KiSc+12}.

Elastic sheets (e.g., polystyrene) range in thickness from several tens to a few hundred nanometers, making them quite challenging to handle. One effective approach to manipulate these sheets is to transfer them from a solid substrate, such as glass, directly onto a liquid surface. In the experiments~\cite{PaDeSa+15, KiSc+12}, the sheet is placed on a liquid drop, which is then gradually reduced in size, thereby increasing the substrate's curvature in a controlled manner.

The mathematical distinction between an elastic substrate (discussed in~\cite{BeKo17}) and a liquid substrate is profound; the different scaling of substrate energy in these scenarios even leads to qualitatively different behaviors already in the macroscopic deformation of the sheet. In the elastic case, the shape converges to a stretched configuration in the radial direction, while in the liquid case, it approaches an isometry of the Euclidean metric. To fully understand the liquid scenario, one must consider the rate of convergence to the isometry of the Euclidean metric, which is not necessary for the elastic substrate.

The aim of this paper is to generalize the findings of~\cite{BeKo17} to incorporate the case of a liquid substrate. To this end, we introduce a more flexible power-law scaling of the substrate energy, allowing us to cover a broad spectrum of scenarios, including both liquid and solid substrates. While these two examples share similar energy scalings, we also identify another regime that accounts for very weak substrates, necessitating a different analysis of the elastic energy for comprehensive understanding.

Finally, we improve upon the previous results~\cite{BeKo17} concerning elastic substrates by replacing the intricate factor in the upper bound $\exp(\sqrt{\log(h)})$ with a more straightforward power-law expression involving logarithms.

There are many areas of material science where the occurrence and properties of microstructure have been successfully explained based on energy minimization: magnetics~\cite{Ot02,GiZw23}, type-I superconductors~\cite{ChCoKoOt08} (see also \cite{R1,R2,RSS} for recent developments in type-II superconductivity), shape-memory alloys~\cite{BeGo15}, diblock copolymers~\cite{ChPe10}, to name just a few. In all of them, one minimizes a non-convex energy regularized with a higher-order term. The non-convex part of the energy favors oscillations, while the higher-order part limits their rate.  
With the regularizing term having a small prefactor (in the present case played by the non-dimensionalized thickness of the sheet), weaker regularization allows for rapid oscillations and gives rise to a microstructure.  

Before we proceed to the actual statements, let us put our result into perspective. Over the last decade, there has been vast activity in studying deformations of thin elastic sheets, in particular their ability to deal with compressive stresses. To allow for experiments with extremely thin sheets with thickness as small as several tens of nanometers, one usually places them on a liquid bath---either flat or curved as in the case of a droplet~\cite{KuRuDaMe20,PaHo++16}. When a flat elastic sheet lies on a curved substrate, stresses arise within the sheet~\cite{DaSuGr19,HoDa15,HuRoBi11}, with the compressive components typically relieved by out‑of‑plane oscillations.
These deformations might be quite large, hence requiring new physical arguments based on nonlinear stability theory~\cite{Ve19}. Recently, new physical arguments exploiting localized delamination to release incompatibility in some cases have been analyzed~\cite{BoDo++23,DaDe21}.  

On the mathematical side, analysis of elastic sheets via scaling laws goes back as far as e.g.~\cite{BBCoDeMu02}. Motivated by physical experiments~\cite{DaSh++11}, Kohn and the first author~\cite{BellaKohnCPAM} identified the scaling law for the first-order term in the energy expansions in the case of an annular thin sheet stretched in the radial direction (see also~\cite{BellaARMA,BeMa23} for identification of the optimal prefactor and corresponding $\Gamma$-convergence result, respectively). While in~\cite{BBCoDeMu02,BoCoMu16,BoCoMu17} the prescribed metric was biaxial compression, the situation for compression in only one direction (the second direction being free or in tension) is completely different~\cite{BeKo14,BeKo17drapes}---not only in terms of the scaling law (power law $h^\frac43$ vs linear), but also in terms of very different constructions of the upper bound as well as arguments for the lower bound, and also in terms of the spatial energy distribution (localized near the boundary vs relatively uniform distribution). Altough we have mentioned some, there are many other examples of prestrained elastic sheets, for example in biological applications~\cite{LeMa22}.  

Let us also mention that, in a work on blistering of thin elastic sheets attached to a shrunk substrate~\cite{BBCoDeMu02}, the incompatibility is caused by prescribed boundary conditions rather than by the metric. A very different model for the deformation of thin elastic sheets (the d‑cone problem) was recently considered by Olbermann~\cite{Ol18,Ol19}, where, by prescribing conical boundary conditions, one asks how the conical singularity is smoothed out in order to obtain finite bending energy. The problem is highly non‑trivial, particularly at the level of the lower bound, since it is not even clear whether the modification of the cone should be of local nature.

The previously mentioned results all dealt with analysis of the next-order term in the energy (excess energy), which describes the properties of microstructure (wrinkling), since the leading-order term is well understood in the flat case via relaxation. Nevertheless, in many instances of the curved situation even the leading-order term is only partially understood; see recent works by Tobasco~\cite{To18,To21}. More precisely, in those works one observes that the analysis of the leading-order term is very different depending on the type of local deformation (e.g., uniaxial or biaxial compression). In the present paper the difficulty is different, since the leading-order term vanishes in the limit, and in order to understand the next-order term we use a thickness-dependent proxy for it.

We also note the connection to the literature on dimension reduction and geometric rigidity. The seminal works of Friesecke, James, and Müller~\cite{frieseckejamesmuller1,frieseckejamesmuller2} established geometric rigidity and derived a hierarchy of plate models via $\Gamma$-convergence. More recent developments extend these ideas to non-Euclidean geometries and prestrained shells~\cite{lewicka1,lewicka2,padilla-garza}.

\subsection{Heuristic arguments}\label{sec:heuristic}
Before we discuss the form of the energy, let us heuristically understand the situation. Assuming first, for simplicity, that the sheet does not stretch in the radial direction, circles at distance $r$ from the pole (center of the sheet) should then map to shorter circles. Put differently, the distance along the sphere from the pole to a circle of radius $r$ is longer than $r$---simply due to the curvature of the sphere. If the sheet is thin enough, it prefers to oscillate out-of-plane rather than compress, which leads to the formation of wrinkles.  

Assuming small deformations of the sheet, we model the elastic energy of the sheet of thickness $h$ using a simplified \emph{F\"oppl-von K\'arman} energy, consisting of a membrane part (measuring deviation of the midplane deformation from being an isometry of the Euclidean metric) plus a bending part which penalizes curvature. The first term has prefactor $h$, which comes from the volume, while the bending is multiplied by a factor $h^3$. We do not assume any connection between the sheet and the substrate; hence the sheet is allowed to flow freely on it. Since we do not allow for cavities (delamination) between the sheet and the substrate, the only energetic contributions from the substrate are related to gravity (pulling up or pushing down the fluid caused by the out-of-plane displacement of the sheet). Normalizing the whole energy by the volume (i.e., dividing by $h$), the substrate prefactor scales like $h^{-1}$. In comparison, in~\cite{BeKo17} the first author and Kohn considered an elastic substrate, in which case the prefactor was of order $h^{-2}$---the exponent coming from the scaling of the $H^{1/2}$-norm of the out-of-plane displacement, such a norm being a proxy for the elastic energy of the substrate (see discussion in~\cite[Section 2]{BeKo17}). For a more general treatment, in the present paper we allow for any scaling $h^{-\beta}$ with $\beta \in (0,2]$ instead of just $\beta = 2$.  

To conform with the results in the physics community, and since we assume we are in the regime of relatively small (at least in-plane) deformations, as announced, instead of nonlinear elasticity we consider a F\"oppl-von K\'arman energy, with a simplified role of the in-plane displacement. Both the sheet and the substrate possess radial symmetry, so it is convenient to rewrite the energy using polar coordinates.  

We denote by $u=(u_r,u_\theta)$ and $\xi$ the in-plane (the radial and hoop parts) and out-of-plane \emph{displacements}, respectively. Further, we denote by $r_0$ and $R$ respectively the radii of the circular sheet and the liquid ball underneath. The F\"oppl-von K\'arman energy in the radial variables then has the form (see equation (3.1) in \cite{BeKo17} in the case $\beta=2$)
\begin{multline*}
E_h(u,\xi)=\int_0^{r_0}\int_0^{2\pi}\Bigg[\left|\partial_r u_r+\frac{(\partial_r \xi)^2}2 \right|^2+\left|\frac{\partial_\theta u_\theta}r + \frac{u_r}r+\frac{(\partial_\theta\xi)^2}{2r^2}\right|^2
+2\left|\frac{\partial_\theta u_r}{2r}+\frac{\partial_r u_\theta}2-\frac{ u_\theta}{2r}+\frac{\partial_r\xi \partial_\theta \xi}{2r}\right|^2\\ 
h^2\left(\frac{|\partial_{\theta\theta}\xi|^2}{r^4}+|\partial_{rr}\xi|^2+\frac{2|\partial_{\theta r}\xi|^2}{r^2}\right)+\frac{\alpha_s}{h^\beta}\left|\xi+\frac{r^2}{2R}\right|^2\Bigg] \dt\ r \dr. 
\end{multline*}
Here, the first three terms represent the membrane energy, which in rectangular coordinates takes the form $|e(u) + \tfrac 12 \nabla \xi \otimes \nabla \xi|^2$, with $e(u) \colonequals \tfrac 12 (\nabla u + \nabla^T u)$ denoting the symmetric part of the gradient. The fourth term, with prefactor $h^2$, corresponds to the bending energy. While these four terms constitute the usual Föppl–von K\'arman energy, the last term accounts for the cost of deforming the substrate. To be consistent with the approximation in \cite{BeKo17}, instead of describing the spherical shape of the substrate by $R - R(1-(r/R)^2)^{\frac12}$, we replace it with the term $\tfrac{r^2}{2R}$, which is valid under the assumption $r \ll R$, a condition satisfied in the relevant physical experiments \cite{KiSc+12,PaDeSa+15}. This smallness condition also heuristically justifies the use of the Föppl–von K\'arman small-slope approximation for the elastic energy, since under that assumption one expects only deformations with small strains.

To keep the analysis tractable, we have chosen not to include surface energy in the present paper. Since this contribution does not depend on thickness, it should play a role similar to that of the substrate. Heuristically, its effect would be to pull the sheet in the radial direction. However, because the sheet is already stretched radially even without surface tension, we do not expect qualitatively different behavior. It would nevertheless be very interesting to rigorously confirm these claims, but to avoid making the analysis overly technical we do not pursue these ideas here.

The parameter $\beta \in (0,2]$ in the prefactor $\alpha_s h^{-\beta}$ plays a decisive role in the problem. To draw a contrast, in~\cite{BeKo17} only the ``critical case'' $\beta=2$ was considered as a proxy for the case of an elastic substrate (see~\cite{BeKo17} for details). In that case the cost of wrinkling (coming essentially from both the bending and the substrate part of the energy) contributes to the leading term of the energy---naively, bending has two derivatives and prefactor $h^2$, the membrane part has one derivative with no $h$, and the substrate term has no derivative and prefactor $h^{-2}$---so, scaling-wise, they match well. In particular, the ``leading order'' (relaxed) energy has a non-trivial form involving a term related to wrinkling (i.e., the cost of wrinkling enters this part of the energy). Hence the limiting shape (while factoring out small-scale oscillations due to wrinkling) will be deformed in the radial direction, i.e., the radial $_{rr}$ part of the stress will not vanish. In particular, in contrast to a ``weaker'' substrate as considered here, the limiting shape is strained in the radial direction---see the form of the minimizer $\u_h^0$ of the relaxed/limiting problem in \eqref{v0} for the case $\beta=2$, which is naturally $h$-independent.

In contrast, in the case $\beta \in (0,2)$, with the liquid substrate ($\beta = 1$) being a special case, the minimum of the ``leading order'' energy $F_h^0$, defined in \eqref{def:F_h}, converges to $0$ (see Remark~\ref{rmk2.1}), and the $h$-dependent minimizers $\u_h^0$ converge to a trivial shape, isometric in the radial direction. Nevertheless, we are able to analyze the behavior of the energy (see Theorem~\ref{thm1} below), which can be seen as the difference between the minimum of the full energy and the one-dimensional $h$-dependent proxy functional $F^0_h$, which describes the leading order energy.  

In addition, in the present paper we also improve the understanding of the critical case $\beta=2$, by replacing the upper bound $C h \exp(C |\log(h)|^{1/2} \log(-\log(h)))$ from~\cite{BeKo17} with the much better bound $C h |\log(h)|^C$.

\subsection{Main results}\label{sec:mainresults}
Our main results provide lower and upper bounds for the minimum of the excess energy (the difference between the minimum of $E_h$ and the leading-order term $F_h^0(\u_h^0)$), for which we naturally need to specify the space $X$ over which the minimization is carried out. For $X$ we choose the largest reasonable space for which the energy makes sense:
\begin{equation}\label{defX}
 X \colonequals \{ (u,\xi) : (0,r_0) \times [0,2\pi] \mapsto \R^2 \times \R, u \in W^{1,2}, \xi \in W^{2,2} \}.
\end{equation}

\begin{theorem}\label{thm1} Let $\beta\in(0,2]$ and $\alpha_s\in \left(0,2^{-8}r_0^4R^{-4}\right)$ if $\beta=2$. There exist positive constants $h_0$, $c_0$, $c_1$, $c_2$, $c_3$, $c_4$ depending on $\alpha_s$, $r_0$, and $R$ (with $h_0$ also depending on $\beta$ if $\beta\leq \frac23$), such that for any $0<h<h_0$, we have 
\begin{equation}\label{mainresulteq}
c_0\expo 
\leq \inf_{(u,\xi)\in X}E_h(u,\xi)-F_h^0(\u_h^0)\leq
\left\{ 
\begin{array}{cl}
c_1|\log h|^{c_2}\expo & \mathrm{if}\ \frac23\leq \beta \leq 2\\
c_3|\log h|^{c_4}\factors & \mathrm{if}\ 0< \beta \leq \frac23.
\end{array}\right.
\end{equation}
\end{theorem}
\begin{remark}
In the ``critical case'' $\beta=2$, this result improves the upper bound contained in \cite{BeKo17}. Nevertheless, the lower bound remains the same, which is why we will assume $\beta<2$ in the proof of the lower bound.
\end{remark}

\begin{remark}
	In the case $\beta=2$, the natural condition $\alpha_s < 2^{-8}r_0^4 R^{-4}$ arises from requiring $r_h < r_0$. The quantity $r_h$ is defined in \eqref{def:r_h} and discussed in Section~\ref{sec:effectivefunctional}, together with the effective functional and its minimizer. If this inequality is not satisfied, wrinkling does not occur in this regime (see \cite{BeKo17} for details).
\end{remark}

\begin{remark}\label{remarkrh3}
    As a technical convenience for the proof of the lower bound in \eqref{mainresulteq}, in Section~\ref{secprooflb} we assume that $r_h \leq \tfrac{1}{3} r_0$. This condition holds automatically when $h$ is sufficiently small for $\beta \in (0,2)$, although the required smallness of $h$ depends on $\beta$ and becomes increasingly restrictive as $\beta \to 2$. To impose the restriction in a way that is independent of $\beta$, we adopt the stronger assumption $\alpha_s < 3^{-6}2^{-8}r_0^4R^{-4}$ when $\beta=2$.
    
    The stronger restriction $r_h \leq \tfrac{1}{3} r_0$ (with $h$ independent of $\beta$) is adopted purely for convenience and does not affect the analysis. With a minor modification of the proof of \eqref{mainresulteq} in Section~\ref{secprooflb}, the general case $r_h < r_0$ can also be treated, though at the expense of having to keep track of the quantity $r_0 - r_h$ throughout the argument.
\end{remark}

\begin{theorem}\label{thm2} Letting $\sigma(\bar u_r,\bar w)\colonequals \bar u_r' +\frac12\left(\frac rR-\bar w'\right)^2$ and defining 
\begin{equation}\label{X}
X_+\colonequals \{(u,\xi)\in X \ | \ \sigma(\bar u_r,\bar w)\geq 0 \},
\end{equation}
if $0< \beta\leq \frac23$, then there exist positive constants $h_0,c_3,c_4$ depending on $\alpha_s$, $r_0$, and $R$ (with $h_0$ also depending on $\beta$), such that for any $0<h<h_0$, we have
\begin{equation}\label{mainresulteq2}
c_0h^\frac43 \leq \inf_{(u,\xi)\in X_+}E_h(u,\xi)-F_h^0(\u_h^0)\leq c_3|\log h|^{c_4}\factors.
\end{equation}
\end{theorem}
\begin{remark}
The condition $\sigma(\bar u_r,\bar w)\geq 0$ appears physically natural, as it indicates that the sheet is slightly stretched in the radial direction. Indeed, due to the curvature of the substrate, one expects the sheet to undergo radial stretching or angular compression—or, more precisely, a combination of both. Nevertheless, we were not able to rigorously establish this condition and therefore had to assume it.
\end{remark}

\section*{Acknowledgments} 
C.R. was partially supported by ANID FONDECYT Regular 1231593. Both authors acknowledge support from the German Research Foundation (DFG) within the Emmy Noether Junior Research Group BE 5922/1-1. Finally, the authors wish to thank the anonymous referee for a careful review and useful suggestions, which have helped improve this paper.

\section{Preliminaries}
This section collects technical definitions and computations that serve as background for the arguments developed later in the paper. They are not required for following the statements of the main results in Section~\ref{sec:mainresults}, except for the definition of the effective functional $F_h^0$ and its minimizer $\u_h^0$ (see Section~\ref{sec:effectivefunctional}). The material is organized into three parts: first, we rewrite the full energy $E_h$ in a convenient form and introduce auxiliary quantities; second, we analyze the energetic cost of elastic deformations of a circle, leading to the reduced functional $W_\rel$; and third, we define the effective functional $F_h^0$ and compute its minimizer explicitly. Readers mainly interested in the main results may skip the first two parts on a first reading, while those wishing to see the detailed structure of the energy may find them useful for context.

\subsection{The energy}
Besides the out-of-plane displacement $\xi$, it will often be convenient to consider the relative out-of-plane displacement  
$$
w(r,\theta)\colonequals \xi(r,\theta)+\frac{r^2}{2R}.
$$
Using polar coordinates, it is convenient to introduce averaging 
\begin{equation}\label{averaging}
\bar f \colonequals \frac{1}{2\pi} \int_0^{2 \pi} f(\theta) \dt \quad \mbox{for }f \in L^1(0,2\pi).
\end{equation}

Using simple manipulations (see the beginning of \cite[Section 3]{BeKo17}, in particular~\cite[(3.5--3.7)]{BeKo17}), we can rewrite the energy $E_h$ in the form
\begin{multline}\label{energy}
E_h(u,w)=\int_0^{r_0}\Bigg(\left|\bar u_r' +\frac12 \left(\frac rR -\bar w' \right)^2+B(r) \right|^2+W_r\left(\frac{\bar u_r}r,w\right)\\ +h^2\left|\bar w''-\frac1R \right|^2+\frac{\alpha_s}{h^\beta}|\bar w|^2\Bigg) r\dr+R_h(u,\xi),
\end{multline}
where  
\begin{equation}\label{defB}
	B(r)\colonequals \overline{|\partial_r(\bar w-w)|^2},
\end{equation}
\begin{equation}\label{defW_r}
W_r(\eta,w)\colonequals \left|\eta+\frac{\overline{|\partial_\theta w|^2}}{2r^2} \right|^2+h^2 \overline{\frac{|\partial_{\theta\theta} w|^2}{r^4}} +\frac{\alpha_s}{h^\beta} \overline{|w-\bar w|^2},
\end{equation}
and the non-negative remainder  
\begin{multline}\label{defR_h}
R_h(u,\xi)=
\int_0^{r_0}\fint_0^{2\pi}\Bigg( \left|\partial_r(u_r-\bar u_r)+\frac{(\partial_r\xi)^2}2-\frac{\overline{(\partial_r\xi)^2}}2 \right|^2
\\+
\left|\frac{\partial_\theta u_\theta}r +\frac{u_r-\bar u_r}r+\frac{(\partial_\theta w)^2}{2r^2}-\frac{\overline{(\partial_\theta w)^2}}{2r^2} \right|^2
\\+
\frac12 \left|\frac{\partial_\theta u_r}r+r\partial_r\left(\frac{u_\theta}r\right)+\frac1r \partial_r\xi\partial_\theta\xi \right|^2
+
h^2\left|\partial_{rr} w \right|^2
+ \frac{2h^2}{r^2}\left|\partial_{\theta r}\xi \right|^2\Bigg) \dt\, r\dr.
\end{multline}
Although the function $W_r(\eta,w)$ also depends on the thickness parameter $h$, to keep the notation simple we will avoid writing this dependence. 

\subsection{Energetic cost of elastic deformations of a circle: analysis of \texorpdfstring{$W_r$}{Wr}}
An essential computation in the analysis is estimating the energetic cost of elastic deformations (with the out-of-plane displacement $\xi$) of a circle being fitted into either smaller or larger space (modeled by $-\eta$).  

A slightly simpler situation, which in turn represents ours quite well, is that of deforming a horizontally positioned straight elastic rope of length $\ell$ to make it occupy a different amount of space. Fixing one of its endpoints and moving the other by an amount $\eta$, the minimal energetic cost (using, as before, the F\"oppl-von K\'arman approximation) of such an elastic deformation is given by
$$
W_\rel(\eta)\colonequals 
\min_{u(0)=0,\ u(\ell)=\eta} \int_0^\ell \Bigg( \left|u'+\frac12 \vartheta'^2\right|^2+h^2|\vartheta''|^2+\alpha_sh^{-\beta}|\vartheta|^2\Bigg)\dx,
$$
where $u$ and $\vartheta$ represent, respectively, the horizontal and vertical deformations.

Observe that given any function $\vartheta$, we can explicitly find a function $w_\vartheta$ for which the energy is minimal among all possible horizontal deformations. The Euler--Lagrange equation for the minimizer $w_\vartheta$ reads $u_\vartheta''+\frac12(\vartheta'^2)'=0$, that is $u_\vartheta'+\frac12 \vartheta'^2=C$ for some constant $C$. Integrating this equality, we find
$u_\vartheta(\ell)+\frac12 \int \vartheta'^2=C$. Replacing $u_\vartheta(\ell)=\eta$, we deduce that our original minimization problem is equivalent to
$$
\min \left(\eta+\frac12 \int_0^\ell \vartheta'^2\dx \right)^2+\int_0^\ell\left( h^2|\vartheta''|^2+\alpha_sh^{-\beta}|\vartheta|^2\right)\dx.
$$
We immediately see that if $\eta\ge 0$, it is optimal to deform the rope only horizontally, that is $\vartheta\equiv 0$. In this case, the optimal energetic cost is equal to $\eta^2$. This is also the case if $\eta$ is slightly negative, similar to the case of buckling of rods. For more negative $\eta$, the wrinkling instability occurs, where the out-of-plane deformation is used as a way to waste some arclength.  

To find the optimal energetic cost in the case $\eta<0$, let us observe that for periodic $\vartheta$, using integration by parts and the Cauchy--Schwarz inequality we get
$$
\int_0^\ell \left(h^2|\vartheta''|^2+\alpha_sh^{-\beta}|\vartheta|^2\right) \dx \geq 2\para \int_0^\ell|\vartheta'|^2\dx.
$$
This inequality being sharp, we see that the minimization problem can be reduced to  
$$
\min_{M\geq 0} \left(\eta+\frac12 M\right)^2+2\para M
$$
with $M=\int_0^\ell \vartheta'^2$. Choosing the optimal $M=\max\bigl(-2\eta-4\para,0\bigr)$, we obtain
\begin{equation}\label{Wrel}
W_\rel(\eta)=\left\{
\begin{array}{cl}
\eta^2&\mbox{if}\ \eta\geq-2\para\\
-4\para(\para+\eta)&\mbox{if}\ \eta\leq-2\para.
\end{array}
\right.
\end{equation}

\subsection{Effective functional}\label{sec:effectivefunctional}
The leading-order behavior of $\min_{u,\xi} E_h$, for $h$ sufficiently small, is obtained by minimizing the \emph{effective functional}
\begin{equation}\label{def:F_h}
F_h^0(\u)\colonequals \int_0^{r_0}\left[\left(\u'+\frac{r^2}{2R^2}\right)^2+W_\rel\left(\frac{\u}r\right)\right]r\dr,
\end{equation}
where $\u$ is a radial real-valued function constrained by the boundary condition $\u(0)=0$ and $W_\rel$ is defined in \eqref{Wrel}. This functional reflects a competition between the elastic energy of radial tension and the elastic cost of deforming circles derived in the previous subsection. We consider $F_h^0$ for two reasons: first, we can explicitly write its minimizer (see~\eqref{v0}); second, it captures the behavior of $E_h$ quite well (Theorem~\ref{thm1}).  

The functional $F_h^0$ is strictly convex, and the unique solution of the corresponding Euler--Lagrange equation has the form  
\begin{equation}\label{v0}
\u_h^0(r)=
\left\{
\begin{array}{cl}
-\dfrac3{16}\dfrac{r^3}{R^2}+\left(2\para \left(\dfrac{r_0}{r_h}-1\right)+\dfrac1{16}\dfrac{r_h^2}{R^2}\right)r&\mbox{for}\ r\in[0,r_h]\\
-2\para r-\dfrac16\left(\dfrac{r^3-r_h^3}{R^2}\right) +2\para r_0 \log\dfrac r{r_h}&\mbox{for}\ r\in [r_h,r_0],
\end{array}
\right.
\end{equation}
where 
\begin{equation}\label{def:r_h}
r_h\colonequals \left(16\para r_0R^2\right)^\frac13.
\end{equation}
Observe that in the first region $(0,r_h)$ we have $\u_h^0(r)/r \ge -2 \para$, i.e., this is exactly the region where $W_\rel$ is quadratic, whereas $W_\rel(\u_h^0(r)/r)$ is affine in $(r_h,r_0)$.  

In the above computation we implicitly assumed that $r_h<r_0$, a condition satisfied for any $h\leq 1$ and any $\beta\in (0,2]$ whenever $\alpha_s<2^{-8}\left(\frac{r_0}R\right)^4$ if $\beta=2$ or for sufficiently small $h$ if $\beta\in (0,2)$. It is convenient to introduce the stress  
\begin{equation}\label{sigma0}
\sigma_h^0\colonequals (\u_h^{0})'+\frac{r^2}{2R^2},
\end{equation}
notation that will be used throughout the rest of the paper.  

\begin{remark}\label{rmk2.1} Using \eqref{v0}, one can compute  
\begin{equation*}
F_h^0(\u_h^0)=Ch^\frac{2-\beta}2+o(h^\frac{2-\beta}2),
\end{equation*}
where $C$ is a constant depending on $\alpha_s,r_0,$ and $R$.
\end{remark}

\section{Lower bound}
\subsection{The functional \texorpdfstring{$F_h$}{Fh}}
We will consider another ``effective'' functional lying between the full energy $E_h$ and the previously defined functional $F_h^0$. While $F_h^0$ only depends on the averaged displacement in the radial direction, the functional $F_h$ also considers averages of the out-of-plane displacement---yet it remains relatively accessible since it is defined on functions of $r$ only. As we will see in the next proposition, $F_h$ and $F_h^0$ are still close to each other (their minimizers and energies differ very little), while $F_h$ is at the same time a much better (at least for analytic reasons) approximation of the full energy $E_h$.  

Motivated by~\eqref{energy}, we define $F_h$ as a functional of $(\u,\w)$, which takes the form
\begin{equation*}
F_h(\u,\w)\colonequals \int_0^{r_0}\left[\left( \u'+\frac12\left(\frac rR -\w'\right)^2\right)^2+W_\rel\left(\frac{\u}r\right)+h^2\left|\w''-\frac1R\right|^2+\alpha_sh^{-\beta}|\w|^2\right]r\dr.
\end{equation*}
Observe that $F_h^0(\u) = F_h(\u,0) - \frac{h^2r_0^2}{2R^2}$ for any admissible $\u$, i.e., $F_h$ is a generalization of the functional $F_h^0$. Moreover, at the level of minimizers, the term $\alpha_s h^{-\beta} |\w|^2$ forces $\w$ to be small in $L^2$, though it still allows for possibly rapid oscillations. For these oscillations to be useful, we would nevertheless need $\w' \sim r/R$ in some sense to allow $\u' \sim 0$, and subsequently $\u \sim 0$, which in turn would force a large $L^2$-norm of $\w$. The following result makes these observations quantitative.

\begin{proposition}
The functional $F_h$ has a unique minimizer $(\u_h,\w_h)$ and there exists $C=C(\alpha_s,r_0,R)$ such that
\begin{equation}\label{difF}
|F_h^0(\u_h^0)-F_h(\u_h,\w_h)|\leq C h^2.
\end{equation}
Moreover, 
\begin{equation}\label{estimateuh}
\|\u_h^0-\u_h\|_{L^\infty\left(\left((\frac14r_h,r_0\right)\right)}\leq Ch^{\frac{2+\beta}{4}}|\log h|^\frac12
\end{equation}
and the stress
$$
\sigma_h(r)\colonequals \u_h' +\frac12\left( \frac rR -\w_h'\right)^2
$$
satisfies, for any $r \in (0,r_0)$,
\begin{equation}\label{sigmahpositive}
\sigma_h(r) \ge 0.
\end{equation}
In addition, for $h$ sufficiently small (independently of $\beta$) we have
\begin{equation}\label{sigmah}
\sigma_h(r)=2\para\left(\frac{r_0}r-1 \right)\quad \mathrm{for\ any}\ r\in \left[\frac12 r_0,r_0\right].
\end{equation}
\end{proposition}

\begin{proof} We proceed in several steps.
	
\begin{enumerate}[label=\textsc{\bf Step \arabic*.},ref=\textsc{\bf Step \arabic*},leftmargin=0pt,labelsep=*,itemindent=*,itemsep=10pt,topsep=10pt]
	
	\item \textbf{Existence and uniqueness of minimizer and proof of \eqref{sigmahpositive}.}  
	We begin by showing that $F_h$ admits a unique minimizer, which satisfies \eqref{sigmahpositive}, by introducing a convex auxiliary functional and analyzing its Euler--Lagrange equation.

	We consider the modified functional
	\begin{multline*}
		\widetilde F_h(\u,\w)\colonequals \int_0^{r_0}\Bigg[\left( \u'+\frac12\left(\frac rR -\w'\right)^2\right)_+^2
		+W_\rel\left(\frac{\u}r\right)+h^2\left|\w''-\frac1R\right|^2+\alpha_sh^{-\beta}|\w|^2\Bigg]r \dr.
	\end{multline*}
	Since $W_\rel$ is a convex function and the function $(s,t) \mapsto (s+t^2/2)^2$ is convex in the region $\{s+t^2/2 \ge 0\}$, we see that $\widetilde F_h$ is convex, and it is strictly convex whenever $\u'+\frac12\left(\frac rR -\w'\right)^2\ge0$. By this convexity, the direct method of the calculus of variations ensures the existence of a minimizer $(\widetilde \u_h,\widetilde \w_h)$. Since $\widetilde F_h(\u,\w) \le F_h(\u,\w)$, we observe that $(\widetilde \u_h,\widetilde \w_h)$ also minimizes $F_h$ provided that we show
	$\widetilde F_h(\widetilde \u_h,\widetilde \w_h) = F_h(\widetilde \u_h,\widetilde \w_h)$. 

    Denoting $\widetilde \sigma \colonequals  (\widetilde \u_h' + \frac 12 (\frac rR - \widetilde \w_h')^2)_+$, the minimizer $(\widetilde \u_h,\widetilde \w_h)$ satisfies the Euler--Lagrange equation $(-2\widetilde\sigma r)' + W_\rel'(\widetilde \u_h/r) = 0$ in the weak sense. Observe that both $\widetilde \u_h'$ and $\widetilde \w_h'$ are weakly differentiable: the latter holds since $\widetilde \w_h \in W^{2,2}(0,r_0)$ (which we observe, in particular,  embeds  into $C^{1,1}(0,r_0)$), while the former follows from the (weak) Euler--Lagrange equation.

    Let us prove that $\widetilde \u_h' + \frac 12 (\frac rR - \widetilde \w_h')^2 \ge 0$ almost everywhere in $(0,r_0)$. Suppose by contradiction that the set
    $$
    A \colonequals \left\{ r \in (0,r_0): \widetilde \u_h'(r) + \frac 12 \left(\frac rR - \widetilde \w_h'(r)\right)^2 < 0 \right\}
    $$
    has positive measure. Then $\widetilde\sigma=0$ almost everywhere on $A$, and testing the (weak) Euler--Lagrange equation with functions supported in $A$ yields $W_\rel'(\widetilde \u_h(r)/r)=0$ almost everywhere on $A$. Since $W_\rel'(t)=0$ iff $t=0$, it follows that $\widetilde \u_h=0$ almost everywhere on $A$. Hence, on $A$ we have
    $\widetilde \u_h' + \frac 12 (\frac rR - \widetilde \w_h')^2 = \frac 12 (\frac rR - \widetilde \w_h')^2 \ge 0$,
    a contradiction. Therefore
    $\widetilde \u_h' + \frac 12 (\frac rR - \widetilde \w_h')^2 \ge 0$ almost everywhere in $(0,r_0)$. In turn, the (weak) Euler--Lagrange equation $(-2 \widetilde \sigma r)' + W_\rel'(\widetilde \u_h/r)=0$
    implies $-2 \widetilde \sigma r \in W^{1,1}(0,r_0)$. By Sobolev embedding, since $\widetilde \w_h'$ is continuous, we deduce that $\widetilde \u_h'$ is also continuous, from which it finally follows that the Euler--Lagrange equation holds in the classical sense and $\widetilde \sigma \geq 0$ for any $r \in (0,r_0)$.

	Having shown $\widetilde \u_h' + \tfrac12\bigl(\tfrac rR - \widetilde \w_h'\bigr)^2 \ge 0$, two facts follow: combined with the strict convexity of $\widetilde F_h$ in the region $\widetilde \sigma \ge 0$, we obtain uniqueness of $(\widetilde \u_h, \widetilde \w_h)$ as the minimizer of $\widetilde F_h$, and hence also of $(\u_h,\w_h)$ as the unique minimizer of $F_h$, which moreover satisfies~\eqref{sigmahpositive}.
	
	\item \textbf{Variational comparison with the reference minimizer $\u_h^0$.}  
	We now compare the minimizer $(\u_h,\w_h)$ of $F_h$ with the reference minimizer $\u_h^0$ of $F_h^0$ by analyzing their optimality conditions: both satisfy vanishing first variations, and the Euler--Lagrange equations allow us to quantify the discrepancy between them through derivative estimates. The main estimate obtained in this step (see \eqref{estimateDwF} below) will be used in the next step to establish \eqref{difF}.
	
	\smallskip
	By the minimality of $\u_h^0$ and $(\u_h,\w_h)$ for $F_h^0$ and $F_h$, respectively, we have
	\begin{equation}\label{derivativesF}
		DF_h(\u_h,\w_h)=0\quad \mbox{and}\quad DF_h^0(\u_h^0)=D_{\u} F_h(\u_h^0,0)=0,
	\end{equation}
	where $DF_h$ (resp. $DF_h^0$) denotes the Fr\'echet derivative of $F_h$ (resp. $F_h^0$) and $D_{\u} F_h$ the Gateaux derivative of $F_h$ in the direction $(\u,0)$.
	
	In addition, for any smooth $\varphi$ we compute
	\begin{equation*}
		D_{\w}F_h(\u_h^0,0)[\varphi]=-2\int_0^{r_0} \sigma_h^0 \frac rR \varphi' r\dr -2\int_0^{r_0} \frac{h^2}{R} \varphi'' r\dr.
	\end{equation*}
	Assuming $\varphi$ is compactly supported, integrating by parts, and using the Euler--Lagrange equation $(2\sigma_h^0 r)'=W_\rel'\bigl(\frac{\u_h^0(r)}r \bigr)$, which comes from the variation of $F_h$ with respect to $\u$, we deduce that
	$$
	D_{\w}F_h(\u_h^0,0)[\varphi]=\frac1R \int_0^{r_0}\left[W_\rel'\left(\frac{\u_h^0}r \right)+2\sigma_h^0\right]\varphi r \dr.
	$$
	A direct computation using~\eqref{v0} shows that
	$$
	\int_0^{r_0}\left|W_\rel'\left(\frac{\u_h^0}r \right)+2\sigma_h^0\right|^2r \dr\leq Ch^{2-\beta},
	$$
	where throughout the proof $C$ denotes a constant depending only on $r_0$, $R$, and $\alpha_s$, which may change from line to line. Using the Cauchy--Schwartz inequality, we then find
	\begin{equation}\label{estimateDwF}
		|D_{\w}F_h(\u_h^0,0)[\varphi]|\leq C h^\frac{2-\beta}2 \left(\int_0^{r_0}|\varphi|^2r\dr\right)^\frac12.
	\end{equation}
	
	\item \textbf{Taylor expansion and energy difference estimates. Proof of \eqref{difF}.}  
	In this step we combine Taylor’s theorem, an auxiliary identity, and the estimate \eqref{estimateDwF} from the previous step to control the difference between $F_h(\u_h,\w_h)$ and $F_h(\u_h^0,0)$, thereby establishing \eqref{difF}.
	
	\smallskip
	We define $f(x,y)\colonequals \left(x+\frac12y^2\right)^2$. For any $x,y,a,b\in \R$, we have
	\begin{multline}\label{identityf}
		f(x,y)-f(a,b)-Df(a,b)[x-a,y-b]\\
		=\left(\left(x+\frac12y^2\right)-\left(a+\frac12b^2\right) \right)^2+\left(a+\frac12b^2\right)(y-b)^2.
	\end{multline}
	Indeed, observe that the first term on the right-hand side comes from the Taylor expansion of $g(t)=t^2$ in the form $t^2 - s^2 - 2s(t-s)=(t-s)^2$, whereas the second term on the right-hand side represents the difference between $Df$ and the gradient of $g$.  
	
	Using the above identity for $f$ and Taylor's theorem with remainder in integral form, we find
	\begin{multline}\label{TaylorFh}
		F_h(\u_h^0,0)-F_h(\u_h,\w_h)-DF_h(\u_h,\w_h)[\u_h^0-\u_h,-\w_h]
		\\=\int_0^{r_0}\Bigg[(\sigma_h^0-\sigma_h)^2+\sigma_h(\w_h')^2+\alpha_sh^{-\beta}\w_h^2+h^2(\w_h'')^2
		\\+\int_{\u_h/r}^{\u_h^0/r}W_\rel''(\eta)\left(\frac{\u_h^0}r- \eta\right)\mathrm d\eta  \Bigg]r\dr
	\end{multline} 
	and
	\begin{multline*}
		F_h(\u_h,\w_h)-F_h(\u_h^0,0)-DF_h(\u_h^0,0)[\u_h-\u_h^0,\w_h]
		\\=\int_0^{r_0}\Bigg[(\sigma_h^0-\sigma_h)^2+\sigma_h^0(\w_h')^2+\alpha_sh^{-\beta}\w_h^2+h^2(\w_h'')^2
		\\+\int_{\u_h^0/r}^{\u_h/r}W_\rel''(\eta)\left(\frac{\u_h}r- \eta\right)\mathrm d\eta  \Bigg]r\dr.
	\end{multline*} 
	Summing these two relations and using \eqref{derivativesF} together with the fact  
	$$\int_{\u_h^0/r}^{\u_h/r}W_\rel''(\eta)\left(\frac{\u_h}r- \eta\right) = 
	-\int_{\u_h/r}^{\u_h^0/r}W_\rel''(\eta)\left(\frac{\u_h}r- \eta\right),$$
	we obtain
	\begin{multline}\label{Dw}
		-D_{\w} F_h(\u_h^0,0)[\w_h]=\int_0^{r_0}\bigg[2(\sigma_h^0-\sigma_h)^2+(\sigma_h^0+\sigma_h)(\w_h')^2+2\alpha_sh^{-\beta}\w_h^2\\+2h^2(\w_h'')^2+\frac{\u_h^0-\u_h}r\int_{\u_h/r}^{\u_h^0/r}W_\rel''(\eta)\mathrm d\eta \bigg]r\dr.
	\end{multline}
	Since $W_\rel''\geq 0$ and both $\sigma_h^0$ and $\sigma_h$ are non-negative, all the terms on the right-hand side of \eqref{Dw} are non-negative. Thus, using \eqref{estimateDwF}, we deduce that
	\begin{equation*}
		h^{-\beta}\int_0^{r_0}\w_h^2 r\dr \leq C h^\frac{2-\beta}2 \left(\int_0^{r_0}\w_h^2 r\dr\right)^\frac12,
	\end{equation*}
	that is
	\begin{equation*}
		\left(\int_0^{r_0}\w_h^2 r\dr\right)^\frac12\leq C h^\frac{2+\beta}2.
	\end{equation*}
	This, together with \eqref{estimateDwF}, gives
	\begin{equation*}
		|D_{\w} F_h(\u_h^0,0)[\w_h]|\leq C h^2.
	\end{equation*}
	Inserting this into \eqref{Dw}, we find
	\begin{multline}\label{estimateh2}
		\int_0^{r_0}\bigg[2(\sigma_h^0-\sigma_h)^2+(\sigma_h^0+\sigma_h)(\w_h')^2+2\alpha_sh^{-\beta}\w_h^2\\+2h^2(\w_h'')^2+\frac{\u_h^0-\u_h}r\int_{\u_h/r}^{\u_h^0/r}W_\rel''(\eta)\mathrm{d}\eta \bigg]r\dr\leq Ch^2.
	\end{multline}
	Observe that for a non-negative function $f$, the inequality $\int_a^b f(r) (b-a) \dr \ge \int_a^b f(r)(b-r) \dr$ holds for any values of $a$ and $b$. Hence the last term on the left-hand side of~\eqref{estimateh2} is greater than or equal to the last term in the right-hand side of~\eqref{TaylorFh}. Comparing the remaining terms, we see that the left-hand side of~\eqref{estimateh2} provides an upper bound for the right-hand side in \eqref{TaylorFh}, and so using \eqref{derivativesF} we deduce that   
	\begin{equation*}
		0\leq F_h(\u_h^0,0)-F_h(\u_h,\w_h)\leq C h^2,
	\end{equation*}
	with the lower bound coming from the optimality of $(\u_h,\w_h)$. This combined with the fact $F_h(\u_h^0,0)=F_h^0(\u_h^0)+h^2\frac{r_0^2}{2R^2}$ yields \eqref{difF}.

	\item \textbf{Estimates for $\w_h$ and gradient differences.}  
	We now derive quantitative bounds for $\w_h$ and for the difference between $\u_h$ and $\u_h^0$ at the gradient level. The main estimate obtained in this step (see \eqref{difgrad} below) will be used in the next step to establish \eqref{estimateuh}.
	
	\smallskip
	From \eqref{estimateh2}, together with the non-negativity of all the terms in the left-hand side, we immediately deduce that
	\begin{equation*}
		\int_0^{r_0} \w_h^2 r \dr\leq C h^{2+\beta}\quad \mathrm{and}\quad \int_0^{r_0} (\w_h'')^2 r \dr\leq C.
	\end{equation*}
	Using the Gagliardo--Nirenberg interpolation inequality, we obtain
	\begin{equation*}
		\int_0^{r_0} \w_h'^4 r \dr\leq C\left( \int_0^{r_0} \w_h^2 r \dr \right)^\frac12 \left( \int_0^{r_0} \w_h''^2 r \dr \right)^\frac32+C \biggl(\int_0^{r_0} \w_h^2 r \dr \biggr)^2 \leq C h^\frac{2+\beta}2
	\end{equation*}
	and  
	\begin{equation*}
		\int_0^{r_0} \w_h'^2 r \dr\leq C\left( \int_0^{r_0} \w_h^2 r \dr \right)^\frac12 \left( \int_0^{r_0} \w_h''^2 r \dr \right)^\frac12+C \int_0^{r_0} \w_h^2 r \dr \leq C h^\frac{2+\beta}2.
	\end{equation*}
	From \eqref{estimateh2}, we also have $\int_0^{r_0}(\sigma_h^0-\sigma_h)^2r\dr\leq Ch^2$, which, combined with the previous two inequalities and $\sigma_h^0 - \sigma_h = ((\u_h^0)' - \u_h') + \frac rR \w_h' - \frac{r^2}{2R^2}(\w_h')^2$, yields
	\begin{equation}\label{difgrad}
		\int_0^{r_0}|(\u_h^0)'-\u_h'|^2r\dr\leq Ch^\frac{2+\beta}2.
	\end{equation}

	\item \textbf{Proof of \eqref{estimateuh}.}  
	In this step we establish \eqref{estimateuh} by exploiting the last term in the left-hand side of~\eqref{estimateh2}. The analysis is split into two cases depending on whether $\u_h(r)$ lies above or below $-2\para r$. The estimate \eqref{difgrad} from the previous step plays a crucial role.
	
	\smallskip
	Heuristically, recalling $W_\rel'' = 2\chi_{(-2\para,\infty)}$, the last term in the left-hand side of~\eqref{estimateh2} controls $|\u_h^0 - \u_h|^2$ whenever $\u_h(r) \ge -2\para r$. If instead $\u_h(r) < -2\para r$, the same term still provides control of $\u_h^0 - \u_h$, provided the ``prefactor'' $\int_{-2\para}^{\u_h^0/r} 2\,\mathrm{d}\eta$ is not too small. To make this rigorous, we introduce the interval
	$$
	A_h\colonequals \left(\frac14r_h,\frac12r_h\right)
	$$
	and split it into the sets $A_h^1$ and $A_h^2$ according to whether $\u_h(r)/r$ lies above or below $-2\para$. More precisely, we define
	\begin{equation*}
		A_h^1\colonequals \left\{r \in  A_h \ \vline \ \frac{\u_h(r)}r\geq -2\para \right\} \quad \mathrm{and}\quad A_h^2\colonequals \left\{r \in  A_h \ \vline \ \frac{\u_h(r)}r<-2\para \right\}.
	\end{equation*}
	Since $W_\rel''(\eta)=2$ for any $\eta\geq -2\para$, and recalling that $\frac{\u_h^0(r)}r\geq -2\para$ for $r\in (0,r_h)$, from \eqref{estimateh2} we deduce that
	\begin{equation}\label{small1}
		\int_{A_h^1}\frac{|\u_h^0-\u_h|^2}r\dr = \frac 12 \int_{A_h^1} \frac{\u_h^0-\u_h}r\biggl(\int_{\u_h/r}^{\u_h^0/r}W_\rel''(\eta)\mathrm{d}\eta\biggr) r\dr\leq Ch^2.
	\end{equation}
	On the other hand, observe that for any $r\in A_h^2$, we have
	\begin{equation*}
		\int_{\u_h/r}^{\u_h^0/r}W_\rel''(\eta)\mathrm{d}\eta=\int_{-2\para}^{\u_h^0/r} 2\, \mathrm{d}\eta=2\left(\frac{\u_h^0}r+2\para\right)\geq Cr_h^2,
	\end{equation*}
	where the last estimate follows from an explicit computation using~\eqref{v0}: since $2\para \ge 0$ (both in the above relation and in definition~\eqref{v0}), we see that 
$\frac{\u_h^0}{r} \ge \frac{1}{16R^2}(r_h^2 - 3r^2) \ge \frac{1}{16R^2}\,\frac{r_h^2}{4}$,
where, in the last inequality, we used $r \le \tfrac{r_h}{2}$.
	
	Combining this with \eqref{estimateh2}, and noting that $\u_h(r) \le -2\para r$ implies that $\u_h \le \u_h^0$ in $A_h^2$, we obtain  
	\begin{align}\label{small2}
		\int_{A_h^2}|\u_h^0-\u_h|\dr &= \int_{A_h^2} \frac{\u_h^0-\u_h}{r} r \dr
		\le C \int_{A_h^2} \frac{\u_h^0-\u_h}{r} \biggl( \frac{1}{r_h^2} \int_{\u_h/r}^{\u_h^0/r}W_\rel''(\eta)\mathrm{d}\eta\,\biggr) r\dr \\  
		&\le C \int_0^{r_0} \frac{\u_h^0-\u_h}{r} \biggl( \frac{1}{r_h^2} \int_{\u_h/r}^{\u_h^0/r}W_\rel''(\eta)\mathrm{d}\eta\,\biggr) r\dr \notag
		\leq C\frac{h^2}{r_h^2},
	\end{align}
	where the penultimate inequality follows from the non-negativity of the integrand.  
	
	From \eqref{small1} and \eqref{small2}, using the Cauchy--Schwarz inequality, we deduce that
	\begin{equation*}
		\int_{A_h}|\u_h^0-\u_h|\dr\leq C r_h\left(\int_{A_h^1}\frac{|\u_h^0-\u_h|^2}r\dr \right)^\frac12 + \int_{A_h^2}(\u_h^0-\u_h)\dr\leq Cr_h h+C\frac{h^2}{r_h^2},
	\end{equation*}
	which implies
	\begin{equation}\label{small3}
		\frac1{|A_h|}\int_{A_h}|\u_h^0-\u_h|\dr\leq Ch.
	\end{equation}
	
	From~\eqref{difgrad} and Hölder's inequality it follows that
	\begin{equation*}
		\int_{\frac14r_h}^{r_0}|\u_h^0\, '-\u_h'|\dr \le \biggl(\int_{\frac14r_h}^{r_0}|\u_h^0\, '-\u_h'|^2 r\dr\biggr)^{\frac12} \biggl( \int_{\frac14r_h}^{r_0} \frac{\dr}{r} \biggr)^{\frac12} \le Ch^{\frac{2+\beta}{4}}|\log(Cr_h)|^{\frac12}.
	\end{equation*}
	Combining this with~\eqref{small3} and the (1-dimensional) Poincaré--Wirtinger inequality  
	\begin{equation*}
		\int_{\frac14r_h}^{r_0}\left|(\u_h^0-\u_h)- \frac1{|A_h|}\int_{A_h}(\u_h^0-\u_h)\dr'  \right|\dr \leq C\int_{\frac14r_h}^{r_0}|\u_h^0\, '-\u_h'|\dr,
	\end{equation*}
	where $C$ does not depend on $h$, we deduce that
	\begin{equation*}
		\int_{\frac14r_h}^{r_0}|\u_h^0-\u_h|\dr 
		\leq 
		\frac{r_0-\frac14r_h}{|A_h|}\int_{A_h}|\u_h^0-\u_h|\dr + C\int_{\frac14r_h}^{r_0}|\u_h^0\, '-\u_h'|\dr \le Ch^{\frac{2+\beta}{4}}|\log h|^{\frac12}. 
	\end{equation*}
	Finally, by the Sobolev embedding of $W^{1,1}$ into $L^\infty(\R)$, we deduce that \eqref{estimateuh} holds.
	
	\item \textbf{Proof of \eqref{sigmah}.}  
	Finally, we establish \eqref{sigmah} by combining \eqref{estimateuh} with the Euler--Lagrange equation satisfied by $\u_h$.
	
	\smallskip
	Explicit computations on $\u_h^0(r)$ for $r>r_h$, combined with \eqref{estimateuh}, allow us to deduce that 
	$$
	\frac{\u_h(r)}{r} \leq -2\para
	$$
	for any $r\in I=\left[\frac{1}{2} r_0, r_0\right]$, provided $h$ is sufficiently small (independently of $\beta$), by an argument analogous to Step~\ref{step1mainthm} in the proof of the lower bound part of Theorem~\ref{thm1}. Using the fact that $\u_h$ satisfies the Euler--Lagrange equation
	\begin{equation*}
		(2\sigma_h r)'=W_\rel'\left(\frac{\u_h}r \right)=-4\para
	\end{equation*}
	in $I$ with $\sigma_h(r_0)=0$, a straightforward computation gives \eqref{sigmah}.  
\end{enumerate}
The proposition is thus proved.
\end{proof}

\subsection{Energy estimates towards the lower bound}

The following proposition provides estimates for several non-negative quantities in terms of the excess energy $\inf E_h - F_h^0(\u_h^0)$. It is therefore a crucial step towards establishing the lower bound in our main results, namely in \eqref{mainresulteq} and \eqref{mainresulteq2}. In fact, the proof will proceed by bounding from below the left-hand side of \eqref{inf1} and \eqref{inf2}.

\begin{proposition}
There exists $C=C(\alpha_s,r_0,R)$ such that, recalling \eqref{averaging}, \eqref{defB}, \eqref{defW_r}, \eqref{defR_h}, \eqref{Wrel}, and \eqref{defX}, we have
\begin{multline}\label{inf1}
\inf_{(u,\xi) \in X}\int_0^{r_0} \biggl(\sigma_hB+W_r\left(\frac{\bar u_r}r,\xi\right)-W_\rel\left(\frac{\bar u_r}r\right) + 
\left( \left(\frac{\u_h}r \vee -2\para\right) - \left(\frac{\bar u_r}r \vee -2\para\right)\right)^2 
\biggr)r\dr\\
+h^\frac{2-\beta}2 \biggl( \int_0^{r_0} (\bar u_r' -\u_h')_+\, r\dr \biggr)^2 +
R_h(u,\xi)\\
\leq C\left(\inf_{(u,\xi) \in X}E_h(u,\xi)-F_h^0(\u_h^0)\right)+O(h^2),
\end{multline}
where hereafter $a\vee b\colonequals \max\{a,b\}$.
Besides, recalling \eqref{X}, we have
\begin{multline}\label{inf2}
\inf_{(u,\xi)\in X_+}\int_0^{r_0} \left(\sigma_hB+W_r\left(\frac{\bar u_r}r,\xi\right)-W_\rel\left(\frac{\bar u_r}r\right)
+\left( \left(\frac{\u_h}r \vee -2\para\right) - \left(\frac{\bar u_r}r \vee -2\para\right)\right)^2 
\right)r\dr\\
+\int_0^{r_0}B^2r\dr+h^\frac{2-\beta}2 \biggl( \int_0^{r_0} (\bar u_r' -\u_h')_+\, r\dr \biggr)^2 +R_h(u,\xi)\\
\leq C\left(\inf_{(u,\xi)\in X_+}E_h(u,\xi)-F_h^0(\u_h^0)\right)+O(h^2).
\end{multline}
\end{proposition}
\begin{proof}
	We begin by noting from \eqref{difF} that
	\begin{equation}\label{difenergies}
	E_h(u,\xi)-F_h^0(\u_h^0)=E_h(u,\xi)-F_h(\u_h,\w_h)+O(h^2),
	\end{equation}
	where $(u,\xi)\in X$ denotes a minimizer of $E_h$.\footnote{Existence of a minimizer follows from the direct method of the calculus of variations.}  
	Rewriting \eqref{energy}, we obtain
	\begin{multline}\label{expan}
		E_h(u,\xi)-F_h(\u_h,\w_h)=
		\\ \tilde F_h(\bar u_r,\bar w, B)-F_h(\u_h,\w_h)+\int_0^{r_0}\left(W_r\left(\frac{\bar u_r}r,\xi\right)-W_\rel\left(\frac{\bar u_r}r\right)\right)r\dr  +R_h(u,\xi),
	\end{multline}
	where
	\begin{align*}
		\tilde F_h(\bar u_r,\bar w, B) &\colonequals 
		\int_0^{r_0}\biggl(\left|\sigma(\bar u_r,\bar w)+B\right|^2+W_\rel\left(\frac{\bar u_r}r\right)+h^2\left|\bar w''-\frac1R \right|^2+\frac{\alpha_s}{h^\beta}|\bar w|^2\Bigg) r\dr,\\
		F_h(\u,\w) &= 
		\int_0^{r_0}\biggl(\left|\sigma_h\right|^2 + W_\rel\left(\frac{\u}r\right) + h^2\left|\w''-\frac1R\right|^2 + \frac{\alpha_s}{h^\beta}|\w|^2\biggr)r\dr.
	\end{align*}
	Recall that both $W_r - W_\rel$ and $R_h$ on the right-hand side of \eqref{expan} are non-negative.  
	
	Using \eqref{identityf}, $DF_h(\u_h,\w_h)=0$, and Taylor’s theorem with remainder in integral form, we obtain
	\begin{multline}\label{ep}
		\ep\colonequals 
		\tilde F_h(\bar u_r,\bar w, B)-F_h(\u_h,\w_h)=\\\int_0^{r_0}\Bigg[(\sigma(\bar u_r,\bar w)-\sigma_h+B)^2+\sigma_h(\bar w'-\w_h')^2+\alpha_sh^{-\beta}(\bar w-\w_h)^2
		\\+h^2(\bar w''-\w_h'')^2+2\sigma_hB+
		\int_{\u_h/r}^{\bar u_r/r}W_\rel''(\eta)\left(\frac{\bar u_r}r- \eta\right)\,\mathrm{d}\eta  \Bigg]r\dr.
	\end{multline}  
	Observe that all terms on the right-hand side are non-negative, in particular due to $\sigma_h \ge 0$ (see \eqref{sigmahpositive}). Hence, from \eqref{difenergies}, \eqref{expan}, and \eqref{ep}, we deduce that, in order to establish \eqref{inf1}, it suffices to prove
	\begin{multline}\label{inf3}
		\int_0^{r_0} \left( \left( \left(\frac{\u_h}r \vee -2\para\right) - \left(\frac{\bar u_r}r \vee -2\para\right)\right)^2 
		\right)r\dr
		+h^\frac{2-\beta}2 \biggl( \int_0^{r_0} (\bar u_r' -\u_h')_+\, r\dr \biggr)^2
		\leq C\ep.
	\end{multline}
	In order to establish the upper bound in \eqref{inf3} for the term involving $\int_0^{r_0} (\bar u_r' -\u_h')_+\, r\dr$, we note that
	\begin{equation*}
		\int_0^{r_0}(\bar w-\w_h)^2r\dr \leq C\ep h^\beta\quad \text{and}\quad \int_0^{r_0}(\bar w''-\w_h'')^2r\dr\leq C\ep h^{-2}.
	\end{equation*}
	By interpolation, we then deduce
	\begin{equation}\label{ineqww}
		\int_0^{r_0}(\bar w'-\w_h')^2r\dr \leq C\ep h^\frac{\beta-2}2,
	\end{equation}
	which implies $\int_0^{r_0}|\bar w'-\w_h'|r\dr\leq C\ep^\frac12h^\frac{\beta-2}4$. Since $B \ge 0$, we conclude that
	\begin{align*}
		\int_0^{r_0} (\bar u_r' - \u_h')_+\, r\dr&\leq \int_0^{r_0}|\bar u_r'-\u_h'+B|\,r\dr\\
		&\leq C\int_0^{r_0}\left(|\sigma(\bar u_r,\bar w)-\sigma_h +B|+\left|\left(\frac rR -\bar w'\right)^2-\left(\frac rR-\w_h'\right)^2\right|\right)r\dr\notag\\
		&\stackrel{\eqref{ep}\&\eqref{ineqww}}{\leq} C(\ep^\frac{1}{2}+\ep^\frac12 h^\frac{\beta-2}4) \leq C \ep^\frac12 h^\frac{\beta-2}4, \notag
	\end{align*}
	where the last inequality uses $\beta\leq 2$.  
	
	Thus, to establish \eqref{inf3}, it remains to prove
	$$
	\int_0^{r_0}\left(\left( \left(\frac{\u_h}r \vee -2\para\right) - \left(\frac{\bar u_r}r \vee -2\para\right)\right)^2\right)r\dr\leq C\ep,
	$$  
	which follows since $W_\rel''(\eta) = 2$ in the region $\{\eta \ge -2\para\}$, and therefore
	$$
	\left( \left(\frac{\u_h}r \vee -2\para\right) - \left(\frac{\bar u_r}r \vee -2\para\right)\right)^2
	$$
	is directly bounded above by the term $\int_{\u_h/r}^{\bar u_r/r}W_\rel''(\eta)\left(\frac{\bar u_r}r- \eta\right)\,\mathrm{d}\eta$  
	contained in \eqref{ep}. This completes the proof of \eqref{inf1}.  
	
	\medskip
	Finally, in the case $(u,\xi)\in X_+$, arguing as before, we deduce that to establish \eqref{inf2}, in view of \eqref{difenergies}, \eqref{expan}, \eqref{ep}, and \eqref{inf3}, it suffices to show
	$$
	\int_0^{r_0}B^2\ r\dr\leq C\ep,
	$$
	which follows from
	$$
	\tilde F_h(\bar u_r,\bar w, B)=F_h(\bar u_r,\bar w)+\int_0^{r_0}(2\sigma(\bar u_r,\bar w)B+B^2)\ r\dr\geq F_h(\u_h,\w_h)+\int_0^{r_0}B^2rdr.
	$$
	This completes the proof of \eqref{inf2} and thus of the proposition. 
\end{proof}

\subsection{Energetic cost of changing the wavenumber}

The following lemma is crucial for identifying the next-order term in the expansion of the energy. At the heuristic level, this next order consists of two main contributions: penalization for not having the optimal number of wrinkles (wavelength), as assumed in the derivation of~\eqref{Wrel}, and the cost of changing this wavenumber as a function of $r$. We remark that the assumption \eqref{assumpetai} on $\eta_i$, stated below, ensures that we will see some wrinkling, since otherwise the above consideration would not hold.

\begin{lemma}\label{lem:Ws}
	Let $\de\geq 0$ and $\rho_0,\rho_1\in (0,r_0)$ with $\rho_0<\rho_1$ be such
	\begin{equation}\label{assumpetai}
		\eta_i\colonequals \frac{\bar u_r(\rho_i)}{\rho_i}\leq -2\para-\de
	\end{equation}
	for $i=0,1$. Then, recalling \eqref{defB}, \eqref{defW_r}, and \eqref{Wrel}, and letting $\la=\rho_1-\rho_0$, we have\footnote{Here, $W_{\rho_i}$ corresponds to $W_r$ evaluated at $r=\rho_i$.} 
	\begin{multline}\label{ineq1}
		\sum_{i=0}^1 (W_{\rho_i}(\eta_i,\xi) - W_\rel(\eta_i)) + \para \fint_{\rho_0}^{\rho_1} B(r) \dr\\
		\geq \frac\de2 \min\left(\frac\de2,\left( \frac{4\rho_1^4}{\rho_0^2} \frac1{\para \la^2}+\frac{(\rho_0+\rho_1)^2}{4\rho_0^2}\frac{\la^2}{h^2}\right)^{-1} \right)
	\end{multline}
	and
	\begin{multline}\label{ineq2}
		\sum_{i=0}^1 (W_{\rho_i}(\eta_i,\xi) - W_\rel(\eta_i)) + \fint_{\rho_0}^{\rho_1}B^2(r)\dr \\
		\geq \frac\de2 \min\Bigg\{\frac\de2,\frac{\para \la^2}{2\rho_0^2},\frac{\para \la^2\rho_0^2}{2\delta \rho_1^2},
		\frac{\de}8\left(\frac{8\rho_1^4}{\rho_0^2}\frac1{\para\la^2}+\frac{\alpha_s(\rho_0+\rho_1)^4}{8\rho_0^4}\frac{\la^4}{h^{2+\beta}}\right)^{-1}
		\Bigg\}.
	\end{multline}
\end{lemma}

\begin{proof} We split the proof in several steps.
	
\begin{enumerate}[label=\textsc{\bf Step \arabic*.},ref=\textsc{\arabic*},leftmargin=0pt,labelsep=*,itemindent=*,itemsep=10pt,topsep=10pt]
\item \textbf{Fourier setup and expression for $W_r$.}
	We denote $\{a_k(r)\}_{k\in \Z}$ the Fourier coefficients of $w(r,\cdot)$. Then, using Plancherel, we have
	$$
	B(r)=\sum_{k\neq 0} (a_k'(r))^2.
	$$
	Moreover, letting
	$$
	A(r)\colonequals \dfrac1{2r^2}\sum_{k\neq 0} a_k(r)^2k^2,
	$$
	directly using the definition of $W_r$, we obtain
	\begin{align*}
		W_r(\eta,\xi)&=|\eta+A(r)|^2+\frac1{r^2}\sum_{k\neq 0} a_k(r)^2 k^2\biggl[\biggl(\frac{h|k|}r -\frac{\alpha_s^\frac12r}{h^\frac\beta2 |k|} \biggr)^2+2\para\biggr]\\
		&=|\eta+A(r)|^2+4\para A(r)+\frac1{r^2}\sum_{k\neq 0} a_k(r)^2 k^2\biggl(\frac{h|k|}r -\frac{\alpha_s^\frac12r}{h^\frac\beta2 |k|} \biggr)^2.
	\end{align*}
	For the remainder of the proof we assume $a_0(r)=0$ for all $r\in[\rho_0,\rho_1]$, since this simplifies the notation without changing the problem. Indeed, the quantity
	$$
	\sum_{k\neq 0} a_k(r)^2 k^2\biggl(\frac{h|k|}{r}-\frac{\alpha_s^\frac12r}{h^{\beta/2}|k|}\biggr)^2
	$$
	is independent of the value of $a_0(r)$, so $a_0$ plays no role in the subsequent analysis. Without this assumption one would have to write $k\neq 0$ at every occurrence, whereas setting $a_0=0$ allows us to simplify the notation considerably without affecting the argument.
	
	\item\label{step2lemma3.1} \textbf{Immediate lower bound from small $A(\rho_i)$ for $i=0$ or $i=1$.}
	Recalling \eqref{Wrel}, for $\eta\le -2\para$ we have
	\begin{equation}\label{difW}
	W_r(\eta,\xi)-W_\rel(\eta)=\bigl|A(r)+(\eta+2\para)\bigr|^2
	+\frac1{r^2}\sum_{k} a_k(r)^2 k^2\biggl(\frac{h|k|}r -\frac{\alpha_s^\frac12r}{h^\frac\beta2 |k|} \biggr)^2.
	\end{equation}
	By hypothesis $\eta_i+2\para\le -\delta$. Hence, if for some $i\in\{0,1\}$ we have $A(\rho_i)\le \frac\delta2$, then, evaluating at $r=\rho_i$, we find
	$$
	\bigl|A(\rho_i)+(\eta_i+2\para)\bigr|\ge \frac{\delta}{2},
	$$
	so the first term on the right-hand side of \eqref{difW} is bounded below by $\frac{\delta^2}{4}$. Since the second term is nonnegative, we obtain
	$$
	W_{\rho_i}(\eta_i,\xi)-W_\rel(\eta_i)\ge \frac{\delta^2}{4},
	$$
	and the desired lower bounds \eqref{ineq1} and \eqref{ineq2} follow immediately. Thus, it remains only to analyze the complementary case
	\begin{equation}\label{remainingcase}
	A(\rho_0)\ge \frac{\delta}{2}\quad\text{and}\quad A(\rho_1)\ge \frac{\delta}{2},
	\end{equation}
	which we treat in the subsequent steps.
	
	\item \textbf{Lower bound from sub-optimal wavelengths.}
	Let $k_i>0$ be the optimal wavenumber defined via
	$$
	\frac{hk_i}{\rho_i} -\frac{\alpha_s^\frac12\rho_i}{h^\frac\beta2 k_i}=0,
	$$
	that is $k_i=\dfrac{\alpha_s^\frac14 \rho_i}{h^\frac{\beta+2}4}$ for $i=0,1$. Let us also define $K\colonequals \dfrac{k_1-k_0}2=\dfrac{\alpha_s^\frac14 \la}{2h^\frac{\beta+2}4}$.
	
	Letting $A_K(\rho_0)=\frac1{2\rho_0^2}\sum_{||k|-k_0|\ge K} a_k(\rho_0)^2 k^2$, which in words corresponds to the energy of wrinkles with sub-optimal wavelengths, we have
	\begin{align*}
		W_{\rho_0}(\eta_0,\xi)-W_\rel(\eta_0)&\geq \frac1{\rho_0^2}\sum_{||k|-k_0|\ge K} a_k(\rho_0)^2 k^2\left(\frac{h|k|}{\rho_0} -\frac{\alpha_s^\frac12 \rho_0}{h^\frac\beta2 |k|} \right)^2\\
		&=\frac1{\rho_0^2}\sum_{||k|-k_0|\ge K} a_k(\rho_0)^2 k^2\frac{h^2}{\rho_0^2}\left(\frac{|k|^2-k_0^2}{|k|}\right)^2\\
		&\geq \frac1{\rho_0^2}\sum_{||k|-k_0|\ge K} a_k(\rho_0)^2 k^2\frac{h^2}{\rho_0^2}||k|-k_0|^2\\
		&\ge 2 \frac{1}{2\rho_0^2} \frac{h^2}{\rho_0^2} K^2 \sum_{||k|-k_0|\ge K} a_k(\rho_0)^2 k^2
		\\
		&\geq 2K^2\frac{h^2}{\rho_0^2}A_K(\rho_0)=\frac12\para \frac{\la^2}{\rho_0^2}A_K(\rho_0),
	\end{align*}
	that is
	\begin{equation}\label{A0}
		A_K(\rho_0)\leq \frac{2\rho_0^2}{\para \la^2}(W_{\rho_0}(\eta_0,\xi)-W_\rel(\eta_0)).
	\end{equation}
	
	\item \textbf{Lower bound from near-optimal wavelengths.}
	Since
	\begin{align*}
		a_k(\rho_0)^2&\leq 2a_k(\rho_1)^2+2|a_k(\rho_1)-a_k(\rho_0)|^2\\
		&\leq 2a_k(\rho_1)^2+2\left(\int_{\rho_0}^{\rho_1} a_k'(r)\dr\right)^2
	\end{align*}
	we have for the energy of wrinkles with close to optimal wavelengths
	\begin{align}
		A(\rho_0)-A_K(\rho_0)&=\frac1{2\rho_0^2}\sum_{||k|-k_0| < K} a_k(\rho_0)^2 k^2 \label{A1}\\
		&\leq \frac{\rho_1^2}{\rho_0^2}\frac1{\rho_1^2}\sum_{||k|-k_0| < K} a_k(\rho_1)^2 k^2+\sum_{||k|-k_0|<  K}  \frac{k^2}{\rho_0^2}\left( \int_{\rho_0}^{\rho_1}a_k'(r) \dr\right)^2.\notag
	\end{align}
	From the definition $K = \frac{k_1 - k_0}{2}$ it follows $\{ k : ||k| - k_0| < K \} \subset \{ k: ||k|-k_1| \ge K\}$, hence
	$$
	\sum_{||k|-k_0| < K} a_k(\rho_1)^2 k^2\leq \sum_{||k|-k_1|\geq K} a_k(\rho_1)^2 k^2\leq \sum_{||k|-k_1|\geq K} a_k(\rho_1)^2 k^2 \frac{||k|-k_1|^2}{K^2}.
	$$
	Combining with
	$$
	||k|-k_1|^2\leq \left(\frac{|k|^2-k_1^2}{|k|}\right)^2 = \frac{\rho_1^2}{h^2}\left(\frac{h|k|}{\rho_1} -\frac{\alpha_s^\frac12 \rho_1}{h^\frac\beta2 |k|} \right)^2,
	$$
	we deduce that
	\begin{align}
		\frac{\rho_1^2}{\rho_0^2}\frac1{\rho_1^2}\sum_{||k|-k_0| < K} a_k(\rho_1)^2 k^2 &\leq \frac{\rho_1^4}{h^2\rho_0^2K^2}\frac1{\rho_1^2}\sum_k a_k(\rho_1)^2 k^2\left(\frac{h|k|}{\rho_1} -\frac{\alpha_s^\frac12 \rho_1}{h^\frac\beta2 |k|} \right)^2\label{A2}\\
		&\leq \frac{4\rho_1^4}{\para \rho_0^2\la^2} (W_{\rho_1}(\eta_1,\xi)-W_\rel(\eta_1)).\notag
	\end{align}
	Besides, by observing that, for any $k$ such that $||k|-k_0|\leq K$, one has
	\begin{equation*}
		|k|\leq \frac{k_0+k_1}2=\frac{\alpha_s^\frac14 (\rho_0+\rho_1)}{2h^\frac{\beta+2}4},
	\end{equation*}
	we deduce that
	\begin{equation}\label{A3}
		\sum_{||k|-k_0| < K}\frac{k^2}{\rho_0^2}\left(\int_{\rho_0}^{\rho_1} a_k'(r)\dr\right)^2\leq \frac{(\rho_0+\rho_1)^2}{4\rho_0^2h^2}\para \sum_{k}\left(\int_{\rho_0}^{\rho_1} a_k'(r)\dr\right)^2.
	\end{equation}
	By plugging \eqref{A2} and \eqref{A3} into \eqref{A1}, we are led to
	\begin{multline}\label{A4}
		A(\rho_0)-A_K(\rho_0) \\ \leq \frac{4\rho_1^4}{\para \rho_0^2\la^2} \left(W_{\rho_1}(\eta_1,\xi)-W_\rel(\eta_1)\right)+ \frac{(\rho_0+\rho_1)^ 2}{4\rho_0^2h^2} \para \sum_k\left(\int_{\rho_0}^{\rho_1} a_k'(r)\dr\right)^2.
	\end{multline}
	
	By Cauchy--Schwarz inequality, we have
	$$
	\left(\int_{\rho_0}^{\rho_1} a_k'(r)\dr\right)^2\leq \lambda \int_{\rho_0}^{\rho_1} (a_k'(r))^2\dr,
	$$
	and therefore
	\begin{equation}\label{A5}
		\sum_k\left(\int_{\rho_0}^{\rho_1} a_k'(r)\dr\right)^2\leq \lambda \int_{\rho_0}^{\rho_1}\sum_k  (a_k'(r))^2 \dr=\lambda \int_{\rho_0}^{\rho_1} B(r) \dr=\lambda^2 \fint_{\rho_0}^{\rho_1}B(r)\dr.
	\end{equation}
	Plugging \eqref{A5} into \eqref{A4}, we then find
	\begin{multline}\label{A7}
		A(\rho_0)-A_K(\rho_0) \\ \leq \frac{4\rho_1^4}{\para \rho_0^2\la^2} \left(W_{\rho_1}(\eta_1,\xi)-W_\rel(\eta_1)\right)+ \frac{(\rho_0+\rho_1)^ 2}{4\rho_0^2h^2} \para \lambda^2 \fint_{\rho_0}^{\rho_1}B(r)\dr.
	\end{multline}

\item \textbf{Proof of \eqref{ineq1}.}
Summing \eqref{A0} and \eqref{A7}, and using $\rho_0<\rho_1$, we deduce that
$$
A(\rho_0)\leq \frac{4\rho_1^4}{\para \rho_0^2\la^2} \sum_{i=0}^1 (W_{\rho_i}(\eta_i,\xi) - W_\rel(\eta_i))+ \frac{(\rho_0+\rho_1)^ 2}{4\rho_0^2h^2} \para \lambda^2 \fint_{\rho_0}^{\rho_1}B(r)\dr,
$$
which, combined with $A(\rho_0)\geq \frac\de2$, yields
\begin{equation*}
	\frac\de2 \left(\frac{4\rho_1^4}{\rho_0^2}\frac1{\para \la^2}+\frac{(\rho_0+\rho_1)^2\la^2}{4\rho_0^2h^2}\right)^{-1}
	\leq \sum_{i=0}^1 (W_{\rho_i}(\eta_i,\xi) - W_\rel(\eta_i)) + \para \fint_{\rho_0}^{\rho_1} B(r) \dr.
\end{equation*}
This, in view of Step~\ref{step2lemma3.1}, completes the proof of \eqref{ineq1}.

\item \textbf{Proof of \eqref{ineq2}.}
Consider first $A_K(\rho_0)\geq \frac{A(\rho_0)}2$,
which, combined with \eqref{remainingcase} and \eqref{A0}, yields
\begin{equation}\label{lem1}
	\frac\de 4 \frac{\para \la^2}{2\rho_0^2}\leq W_{\rho_0}(\eta_0,\xi)-W_\rel(\eta_0).
\end{equation}

Now assume instead $A_K(\rho_0) < \frac{A(\rho_0)}2$, so that, by \eqref{remainingcase}, we have
\begin{equation}\label{eqqqq}
	A(\rho_0)-A_K(\rho_0)\geq \frac{A(\rho_0)}2\geq \frac\delta4.
\end{equation}
We next distinguish two cases. If
$$
1\leq \frac{4\rho_1^4}{\para \rho_0^2\la^2} (W_{\rho_1}(\eta_1,\xi)-W_\rel(\eta_1))
$$
then
\begin{equation}\label{lem2}
	\frac{\para \rho_0^2\la^2}{4\rho_1^4}\leq W_{\rho_1}(\eta_1,\xi)-W_\rel(\eta_1).
\end{equation}
Otherwise, if
\begin{equation}\label{reverse}
	1> \frac{4\rho_1^4}{\para \rho_0^2\la^2} (W_{\rho_1}(\eta_1,\xi)-W_\rel(\eta_1)),
\end{equation}
then by squaring \eqref{A7} and using \eqref{eqqqq}, we find
\begin{align*}
	\frac{\de^2}{16}
	&\le
	2\biggl(\frac{4\rho_1^4}{\para \rho_0^2\la^2} (W_{\rho_1}(\eta_1,\xi)-W_\rel(\eta_1))\biggr)^2+
	\frac{(\rho_0+\rho_1)^ 4}{8\rho_0^4h^4} \alpha_s h^{2-\beta} \lambda^4 \left(\fint_{\rho_0}^{\rho_1}B(r)\dr\right)^2\\
	&\le 
	\frac{8\rho_1^4}{\para \rho_0^2\la^2} \left(W_{\rho_1}(\eta_1,\xi)-W_\rel(\eta_1)\right)+
	\frac{(\rho_0+\rho_1)^ 4}{8\rho_0^4h^{2+\beta}} \alpha_s \lambda^4 \fint_{\rho_0}^{\rho_1}B^2(r)\dr,
\end{align*}
where in the last inequality we used \eqref{reverse} and the Cauchy--Schwarz inequality. From this we obtain
$$
\frac{\de^2}{16}\left(\frac{8\rho_1^4}{\para \rho_0^2\la^2}+\frac{\alpha_s(\rho_0+\rho_1)^4\la^4}{8\rho_0^4h^{2+\beta}}\right)^{-1}
\leq  (W_{\rho_1}(\eta_1,\xi)-W_\rel(\eta_1))+\fint_{\rho_0}^{\rho_1} B^2(r)\dr,
$$
which, combined with \eqref{lem1} and \eqref{lem2}, yields
\begin{multline*}
	\frac{\delta}2\min\Bigg\{\frac{\para \la^2}{2\rho_0^2},\frac{\para \la^2\rho_0^2}{2\delta \rho_1^2},
	\frac{\de}8\left(\frac{8\rho_1^4}{\rho_0^2}\frac1{\para\la^2}+\frac{\alpha_s(\rho_0+\rho_1)^4}{8\rho_0^4}\frac{\la^4}{h^{2+\beta}}\right)^{-1}
	\Bigg\}\\
	\leq \sum_{i=0}^1(W_{\rho_i}(\eta_i,\xi)-W_\rel(\eta_i))+\fint_{\rho_0}^{\rho_1} B^2(r)\dr.
\end{multline*}
This, in view of Step~\ref{step2lemma3.1}, completes the proof of \eqref{ineq2}.
\end{enumerate}
The lemma is thus proved.
\end{proof}

\begin{remark}
	Provided $\de>0$ and $h$ is sufficiently small, the minimum in \eqref{ineq1} is obtained by balancing the terms depending on $h$ and $\la$, namely by choosing $\la^4=C_1h^\frac{2+\beta}2$. This yields
	\begin{equation}\label{ineqlb1}
		\sum_{i=0}^1 (W_{\rho_i}(\eta_i,\xi) - W_\rel(\eta_i)) + \para \fint_{\rho_0}^{\rho_1} B(r) \dr \geq C_2h^\frac{6-\beta}4.
	\end{equation}
	
	Similarly, the minimum in \eqref{ineq2} is obtained by balancing
	$$
	\frac{8\rho_1^4}{\rho_0^2}\frac1{\para\la^2}\quad\text{and}\quad \frac{\alpha_s(\rho_0+\rho_1)^4}{8\rho_0^4}\frac{\la^4}{h^{2+\beta}},
	$$
	which leads to $\la=C_1h^\frac{2+3\beta}{12}$ and
	\begin{equation}\label{ineqlb2}
		\sum_{i=0}^1 (W_r(\eta_i,\xi) - W_\rel(\eta_i)) + \fint_{\rho_0}^{\rho_1}B^2(r)\dr \geq C_2 h^\frac43.
	\end{equation}
	
	Combining \eqref{ineqlb1} and \eqref{ineqlb2}, we deduce
	\begin{multline*}
		\sum_{i=0}^1 (W_r(\eta_i,\xi) - W_\rel(\eta_i)) + \para \fint_{\rho_0}^{\rho_1} B(r) \dr + \fint_{\rho_0}^{\rho_1}B^2(r) \dr \\
		\geq
		C\max\{\expos,\expo\}
		=C
		\left\{ 
		\begin{array}{cll}
			\expo & \mathrm{if}\ \frac23 \leq \beta<2 & (\mathrm{with}\ \la\sim h^\frac{2+\beta}8)\\
			\expos & \mathrm{if}\ 0 <\beta\leq \frac23 & (\mathrm{with}\ \la\sim h^\frac{2+3\beta}{12}).
		\end{array}
		\right.
	\end{multline*}
	Here $C,C_1,C_2$ depend only on $\rho_0,\rho_1,\alpha_s$, and $\de$.
\end{remark}

\subsection{Proof of the lower bound part of the main results}\label{secprooflb}
\begin{proof}[Proof of the lower bound part of Theorem \ref{thm1}] Let 
\begin{equation*}
I_1\colonequals \left[\frac 12 r_0, \frac7{12}r_0\right] \quad\textrm{and}\quad
I_2\colonequals \left[\frac 23 r_0, r_0\right],
\end{equation*}
where, since $r_h \le \frac13 r_0$ (recall Remark~\ref{remarkrh3}), both are included in the interval $[r_h,r_0]$.  

\begin{enumerate}[label=\textsc{\bf Step \arabic*.},ref=\textsc{\arabic*},leftmargin=0pt,labelsep=*,itemindent=*,itemsep=10pt,topsep=10pt]

\item\label{step1mainthm} \textbf{Verification of the wrinkling hypothesis for Lemma \ref{lem:Ws}.}  
To be able to use Lemma~\ref{lem:Ws}, we need to show that for $r$ at least in $I_2$ the value $\bar u_r(r)/r$ lies well below $-2\para$ (i.e., the sheet should wrinkle there, and not only marginally).

More precisely, we claim that there exist $h_0,\de_0>0$ (where only $h_0$ depends on $\beta$ in the case $\beta\in\left(0,\frac23\right]$) such that, for any $0<h<h_0$ and for almost every $r\in I_2$,
\begin{equation}\label{de0}
\frac{\bar u_r(r)}{r}\leq -2\para-\de_0.
\end{equation}
Before giving the argument, let us explain the reasoning. By direct computation, we know that~\eqref{de0} is true if we replace $\bar u_r$ by $\u_h^0$. By~\eqref{estimateuh}, we know that $\u_h^0$ and $\u_h$ do not differ much, i.e., this also holds for $\u_h$, and it remains to compare $\u_h$ and $\bar u_r$. For that we use~\eqref{inf1}, specifically the control on $\int_0^{r_0} (\bar u_r' - \u_h')_+ r \dr$, which for convenience we replace by $\int_{\frac13 r_0}^{r_0} (\bar u_r' - \u_h')_+ \dr$. In particular, we have for any interval $[a,b] \subset \left[\frac13r_0,r_0\right]$ that 
\begin{equation*}
 (\bar u_r - \u_h)(b) - (\bar u_r - \u_h)(a) = \int_a^b (\bar u_r' - \u_h')(r) \dr \le \int_a^b (\bar u_r' - \u_h')_+(r) \dr.
\end{equation*}
Rearranging this and using the smallness of the r.h.s., we obtain
\begin{equation*}
 \bar u_r(a) - \bar u_r(b)\ge \u_h(a) - \u_h(b) - \textrm{small term}. 
\end{equation*}
Since $\u_h(a) - \u_h(b)$ is bounded above away from $0$, we see that either $\bar u_r(b)$ is below $-2\para b- \de_0b$, or otherwise $\bar u_r(a)$ is much larger than $-2 \para a$. Since the latter cannot be true for many points $a \in I_1$ (otherwise the last term on the first line in~\eqref{inf1} becomes too large), we obtain the conclusion.

\medskip
The rigorous argument proceeds as follows. First, we provide two estimates concerning $\u_h^0$, obtained from explicit computations. We recall that, for any $r\in [r_h,r_0]$, we have
$$
\u_h^0(r)=-2\para r+g(r),
$$
where $g(r)\colonequals -\dfrac16\left(\dfrac{r^3-r_h^3}{R^2}\right) +2\para r_0 \log\dfrac r{r_h}$. Note that, for any $s\in (r_h,r_0)$,
$$
g'(s)=\frac{-s^3+4\para r_0R^2}{2sR^2}.
$$
In particular, recalling that $r_h=(16\para r_0R^2)^\frac13$, we obtain
$$
g'(s)<\frac{-r_h^3+4\para r_0R^2}{2sR^2}<-6\para<0.
$$
Since $g(r_h)=0$, we conclude that  
\begin{equation*}
g(r)< 0 \quad \mbox{for any }r>r_h.
\end{equation*}
Moreover, recalling the assumption $\alpha_s\leq 2^{-8}3^{-6}r_0^4R^{-4}$ in the case $\beta=2$ (see Remark~\ref{remarkrh3}), we see that, for any $h$ sufficiently small (independently of $\beta$), we have
$$
\para r_0R^2\leq 2^{-4}3^{-3}r_0^3.
$$
Hence, for any $s\in \left[\frac13 r_0,r_0\right)$,
$$
g'(s)<\frac{-3^{-3}r_0^3+4\para r_0R^2}{2sR^2}\leq -\frac{3^{-3}r_0^3(1-2^{-2})}{2sR^2}\leq -\frac{r_0^2}{2^33^2R^2}\equalscolon -c_0.
$$
Therefore, for any $r_1,r_2\in \left[\frac13 r_0,r_0\right]$ with $r_1<r_2$, we have
$$
g(r_2)-g(r_1)=\int_{r_1}^{r_2}g'(s)\mathrm{d}s\leq -c_0(r_2-r_1). 
$$
In particular, taking $r_1=\frac{r_0}3\geq r_h$ and recalling that $g(r_1)\leq g(r_h)=0$, we are led to
$$
g(r_2)\leq-c_0\left(r_2-\frac{r_0}3\right).
$$
We immediately deduce that, for a.e. $r_1,r_2\in \left[\frac13 r_0,r_0\right]$ with $r_1<r_2$,
\begin{equation}\label{u_01}
\u_h^0(r_2)-\u_h^0(r_1)\leq -2\para (r_2-r_1)- c_0(r_2-r_1)
\end{equation}
and
\begin{equation*}
\u_h^0(r_2)\leq -2\para r_2-c_0\left(r_2-\frac{r_0}3\right).
\end{equation*}
From this last inequality, we deduce that, for a.e. $r\in \left[\frac12 r_0,r_0\right]$, we have
$$
\frac{\u_h^0(r)}{r}\leq -2\para -\delta_1,\quad \mathrm{where}\ \delta_1= c_0\left(1 -\frac23\right)=\frac{c_0}3. 
$$
By \eqref{estimateuh}, we find
\begin{equation}\label{uh1}
\frac{\u_h(r)}{r}\leq -2\para -\delta_1+O\left(h^{\frac12}|\log h|^\frac12\right)\leq  -2\para -\frac12\delta_1,
\end{equation}
provided $h$ is sufficiently small (independently of $\beta$).

\medskip
On the other hand, for a.e. $r_1,r_2\in \left[\frac12 r_0,r_0\right]$ with $r_1<r_2$, we have  
\begin{equation*}
 (\bar u_r - \u_h)(r_2) - (\bar u_r - \u_h)(r_1) = \int_{r_1}^{r_2} (\bar u_r' - \u_h')(r) \dr \le \int_{r_1}^{r_2} (\bar u_r' - \u_h')_+(r) \dr.
\end{equation*}
In particular, for a.e. $r_1\in I_1$ and a.e. $r_2\in I_2$, we have
$$
\bar u_r(r_2)\leq (\u_h(r_2)-\u_h(r_1))+\bar u_r(r_1)+ \int_{r_1}^{r_2} (\bar u_r' - \u_h')_+(r) \dr.
$$
From \eqref{estimateuh} and \eqref{u_01}, we are led to
$$
\u_h(r_2)-\u_h(r_1)\leq \u_h^0(r_2)-\u_h^0(r_1)+O\left(h^{\frac12}|\log h|^\frac12\right)\leq -2\para (r_2-r_1)-\frac{c_0}2(r_2-r_1),
$$
provided $h$ is sufficiently small (independently of $\beta$). Therefore,
\begin{equation}\label{est1}
\bar u_r(r_2)\leq -2\para (r_2-r_1)-\frac{c_0}2(r_2-r_1)+\bar u_r(r_1)+ \int_{r_1}^{r_2} (\bar u_r' - \u_h')_+(r) \dr.
\end{equation}

\medskip
From \eqref{inf1} and the upper-bound construction \eqref{toproveupperbound}, we deduce that
$$
\int_{I_1} \left( \left(\frac{\u_h(r)}r \vee -2\para\right) - \left(\frac{\bar u_r(r)}r \vee -2\para\right)\right)^2 r \dr \leq Ch^\frac14,
$$
provided $h$ is sufficiently small (independently of $\beta$). Since $I_1\subset \left[\frac12r_0,r_0 \right]$, from \eqref{uh1} we deduce that for a.e. $r\in I_1$,  
$$
\frac{\u_h(r)}r \vee -2\para=-2\para.
$$
Letting $\delta_2 \colonequals \dfrac{c_0}2\left(\dfrac23-\dfrac7{12}\right)=\dfrac{c_0}{24}$ and  
$$
I_1^\mathrm{bad}\colonequals \left\{r\in I_1 \ | \ \frac{\bar u_r(r)}{r}\geq -2\para +\frac1{2}\delta_2\right\},
$$
we deduce that
\begin{multline*}
|I_1^\mathrm{bad}|\left(\frac12\delta_2\right)^2\frac12{r_0} \leq \int_{I_1^{\mathrm{bad}}}\left(\frac{\bar u_r(r)}r +2\para\right)^2r\dr
\\ \leq \int_{I_1} \left( \left(\frac{\u_h}r \vee -2\para\right) - \left(\frac{\bar u_r}r \vee -2\para\right)\right)^2 r \dr \leq Ch^\frac14.
\end{multline*}
In particular, provided $h$ is sufficiently small (independently of $\beta$), we have
$$
|I_1\setminus I_1^{\mathrm{bad}}|\geq \frac{|I_1|}2=\frac1{24}r_0.
$$
Then, by choosing $r_1\in I_1\setminus I_1^{\mathrm {bad}}$, we obtain  
$$
\bar u_r(r_1)\leq -2\para r_1+\frac12 \delta_2r_1\leq -2\para r_1+\frac12\delta_2r_0.
$$
Inserting this and $-\dfrac{c_0}2(r_2-r_1)\leq -\delta_2r_0$ into  
\eqref{est1}, we find, for a.e. $r\in I_2$,
\begin{align*}
\bar u_r(r_2)&\leq -2\para (r_2-r_1)-\delta_2 r_0-2\para r_1+\frac12\delta_2r_0+ \int_{r_1}^{r_2} (\bar u_r' - \u_h')_+(r) \dr\\
&=-2\para r_2-\frac12\delta_2 r_0+\int_{r_1}^{r_2} (\bar u_r' - \u_h')_+(r) \dr,
\end{align*}
hence
\begin{multline}\label{est2}
\frac{\bar u_r(r_2)}{r_2}\leq -2\para -\frac12\delta_2+\frac1{r_2}\int_{r_1}^{r_2} (\bar u_r' - \u_h')_+(r) \dr\\ \leq -2\para -\frac12\delta_2+\frac3{2r_0}\int_{\frac12r_0}^{r_0} (\bar u_r' - \u_h')_+(r) \dr.
\end{multline}
To finish the argument, we combine \eqref{inf1} with the upper-bound construction \eqref{toproveupperbound}, which yields
\begin{equation*}
\frac3{2r_0}\int_{\frac12r_0}^{r_0} (\bar u_r' - \u_h')_+(r)\dr\leq C\left\{\begin{array}{cl}
|\log h|^\frac{c_2}2h^\frac{2+\beta}8 & \mathrm{if}\ \frac23\leq \beta \leq 2\\
|\log h|^\frac{c_4}2h^\frac{\beta}2 & \mathrm{if}\ 0<\beta \leq \frac23.
\end{array}\right.
\end{equation*}
The right-hand side of this expression can be made smaller or equal to $\frac{\delta_2}4$, provided $h$ is sufficiently small (independently of $\beta$ only in the case $\frac23<\beta \leq 2$). Inserting this into \eqref{est2}, we obtain \eqref{de0} with $\delta_0=\frac14\delta_2$.

\item\label{step2mainthm} \textbf{Conclusion.}
Let $I_h=(a,b)$ be an interval of length $2h^\frac{2+\beta}8$ contained in $I_2$. Define $I_h^0=\left(a,a+\frac12 h^\frac{2+\beta}8\right)$ and $I_h^1=\left(b-\frac12h^\frac{2+\beta}8,b\right)$.  

Let $\rho_i$ be such that
\begin{equation*}
W_{\rho_i}\left(\frac{\bar u_r(\rho_i)}{\rho_i},w(\rho_i,\cdot) \right) - W_\rel\left(\frac{\bar u_r(\rho_i)}{\rho_i}\right)=
\min_{r\in I_h^i} W_r\left(\frac{\bar u_r(r)}r,w(r,\cdot)\right) - W_\rel\left(\frac{\bar u_r(r)}r\right)
\end{equation*}
for $i=0,1$. Observe that $\la=\rho_1-\rho_0\in \left[\frac{|I_h|}2,|I_h|\right)=[h^\frac{2+\beta}8,2h^\frac{2+\beta}8)$. The choice of $\rho_0$ and $\rho_1$ then implies
\begin{multline*}
\int_{I_h} \left[W_r\left(\frac{\bar u_r(r)}r,w(r,\cdot)\right) - W_\rel\left(\frac{\bar u_r(r)}r\right) +\para B(r)\right]\dr\\
\geq \frac{|I_h|}4 \sum_{i=0}^1 \left(W_{\rho_i}\left(\frac{\bar u_r(\rho_i)}{\rho_i},w(\rho_i,\cdot)\right) - W_\rel\left(\frac{\bar u_r(\rho_i)}{\rho_i}\right)\right)+\la \para \fint_{\rho_0}^{\rho_1} B(r) \dr.
\end{multline*}

Using \eqref{ineqlb1}, we deduce that
\begin{equation*}
\int_{I_h} \left[W_r\left(\frac{\bar u_r(r)}r,w(r,\cdot)\right) - W_\rel\left(\frac{\bar u_r(r)}r\right) +\para B(r)\right]\dr\geq C_0|I_h|h^\frac{6-\beta}4,
\end{equation*}
where $C_0$ denotes a constant depending only on $\alpha_s$, $r_0$, and $\de_0$.

By covering the interval $I_2$ with intervals like $I_h$ considered above, and using \eqref{sigmah} (recall that $I_2\subset \left[\frac12r_0,r_0\right]$), we deduce that  
\begin{equation*}
\int_{I_2} \left[W_r\left(\frac{\bar u_r(r)}r,w(r,\cdot)\right) - W_\rel\left(\frac{\bar u_r(r)}r\right) +\sigma_h B(r)\right]\dr\geq C_0h^\frac{6-\beta}4.
\end{equation*}
Finally, combining the previous estimate with \eqref{inf1}, we obtain the lower bound in \eqref{mainresulteq}. The proof is thus complete.
\end{enumerate}
\end{proof}

\begin{proof}[Proof of the lower-bound part of Theorem \ref{thm2}] 
We just need to modify Step~\ref{step2mainthm} in the previous proof for $\beta\in (0,\frac23)$. Let $I_h=(a,b)$ be an interval of length $2h^\frac{2+3\beta}{12}$ contained in $I_2$. Define $I_h^0=\left(a,a+\frac12 h^\frac{2+3\beta}{12}\right)$ and $I_h^1=\left(b-\frac12h^\frac{2+3\beta}{12},b\right)$.  

Let $\rho_i$ be such that
\begin{equation*}
W_{\rho_i}\left(\frac{\bar u_r(\rho_i)}{\rho_i},w(\rho_i,\cdot)\right) - W_\rel\left(\frac{\bar u_r(\rho_i)}{\rho_i}\right)=
\min_{r\in I_h^i} W_r\left(\frac{\bar u_r(r)}r,w(r,\cdot)\right) - W_\rel\left(\frac{\bar u_r(r)}r\right)
\end{equation*}
for $i=0,1$. Observe that $\la=\rho_1-\rho_0\in \left[\frac{|I_h|}2,|I_h|\right)=[h^\frac{2+3\beta}{12},2h^\frac{2+3\beta}{12})$. Hence  
\begin{multline*}
\int_{I_h} \left[W_r\left(\frac{\bar u_r(r)}r,w(r,\cdot)\right) - W_\rel\left(\frac{\bar u_r(r)}r\right) +B^2(r)\right]\dr\\
\geq \frac{|I_h|}4 \sum_{i=0}^1 \left(W_{\rho_i}\left(\frac{\bar u_r(\rho_i)}{\rho_i},w(\rho_i,\cdot)\right) - W_\rel\left(\frac{\bar u_r(\rho_i)}{\rho_i}\right)\right)+\la\fint_{\rho_0}^{\rho_1} B^2(r) \dr.
\end{multline*}

Using \eqref{ineqlb2}, we deduce that
\begin{equation*}
\int_{I_h} \left[W_r\left(\frac{\bar u_r(r)}r,w(r,\cdot) \right) - W_\rel\left(\frac{\bar u_r(r)}r\right) +B^2(r)\right]\dr\geq C_0|I_h|h^\frac43,
\end{equation*}
where $C_0$ denotes a constant depending only on $\alpha_s$, $r_0$, and $\de_0$.

By covering the interval $I_2$ with intervals like $I_h$ considered above, we deduce that  
\begin{equation*}
\int_{I_2} \left[W_r\left(\frac{\bar u_r(r)}r,w(r,\cdot)\right) - W_\rel\left(\frac{\bar u_r(r)}r\right) +B^2(r)\right]\dr\geq C_0h^\frac43.
\end{equation*}
Combining this with \eqref{inf2}, we obtain the lower bound in \eqref{mainresulteq2}. The proof is thus complete.
\end{proof}

\section{Upper bound}

In this section we construct an explicit wrinkling ansatz that achieves the upper bound.  
The deformation $(u,w)$ wastes the correct arclength, cancels the leading oscillatory errors,  
and balances the remaining terms by a suitable choice of parameters.  
The resulting energy matches the lower bound scaling in the case $\frac23\leq \beta\leq 2$, up to a logarithmic error, with improved accuracy for $\beta=2$ compared to \cite{BeKo17}.

\subsection{Strategy for the upper bound}
To establish the upper bound, our goal is to construct for each $h$ a deformation $(u,w)$ such that, for some positive constants $c_i$, $i=1,\dots,4$, we have
\begin{equation}\label{toproveupperbound}
	E_h(u,w)-F_h^0(\u_h^0)\leq
	\left\{ 
	\begin{array}{cl}
		c_1|\log h|^{c_2}\expo & \mathrm{if}\ \frac23\leq \beta \leq 2\\
		c_3|\log h|^{c_4}\factors & \mathrm{if}\ 0< \beta \leq\frac23.
	\end{array}\right.
\end{equation}

The construction follows the same guiding principles as in~\cite{BeKo17}: we modify the relaxed minimizer $\u_h^0$ by introducing oscillations (wrinkles) that waste the appropriate amount of arclength. Unlike the case $\beta=2$ treated there, here we must adapt the choice of parameters to handle the full range $0\leq \beta \leq 2$, which requires a careful balance between the bending term $B$, the discrepancy $W_r-W_\rel$, and the remainder contributions.

The basic idea is to define $w(r,\theta)$ so that the excess length $\gamma(r)$ produced by wrinkling matches the target $\gamma_0(r)$ dictated by $\u_h^0$. The amplitudes and frequencies of the wrinkles are chosen so that $\gamma(r)-\gamma_0(r)$ remains small, while the additional cost in the energy is controlled by suitable cut-off functions. A discrete-to-continuous comparison shows that this discrepancy is exponentially small. To deal with the error terms that arise in the expansion of the remainder, we introduce two correctors: a radial component $u_r(r,\theta)$ and an angular component $u_\theta(r,\theta)$. These are chosen explicitly so that their derivatives cancel the leading oscillatory contributions, thereby reducing $R_h$ to controlled terms.

Three parameters play a central role:
\begin{itemize}
	\item $\ell$: the characteristic radial length scale over which the wrinkling pattern varies;
	\item $\alpha$: the exponent controlling the transition region near $r_h$, which determines how quickly the amplitude $A(r)$ grows from zero to its full value;
	\item $\delta$: the exponent governing the wrinkling frequency $N=h^\delta/\ell$, which ensures that the discrete-to-continuous discrepancy in $\gamma(r)-\gamma_0(r)$ is exponentially small.
\end{itemize}  

The parameters $\ell,\alpha,\delta$ are tuned according to the value of $\beta$, so that the three main contributions (bending, mismatch, and remainder) are balanced. The precise formulas are given later, but the guiding principle is that for $\tfrac23\leq\beta\leq 2$ the construction parallels~\cite{BeKo17} with improved error control, thanks to the optimized discrete-to-continuous comparison, while for $0<\beta<\frac23$ the dominant contribution comes from one of the terms in the remainder $R_h$, leading to the rate $\factors$ in this regime.

As in~\cite{BeKo17}, it is sufficient to carry out the construction for a discrete sequence of values of $h$ tending to $0$, chosen so that the frequency conditions are satisfied. This restriction does not affect the asymptotic upper bound.

\subsection{Sketch of the proof}
The implementation of the strategy proceeds in the following steps.

\begin{enumerate}[label=\textsc{\bf Step \arabic*.},ref=\textsc{\arabic*},leftmargin=0pt,labelsep=*,itemindent=*,itemsep=10pt,topsep=10pt]
	
	\item \textbf{The ansatz.}  
    We begin by introducing the deformation $(u,w)$ of the form
    $$
    (u,w)=\Big(\bigl(\u_h^0(r)+u_r(r,\theta),\,u_\theta(r,\theta)\bigr),\,w(r,\theta)\Big),
    $$
    where $\u_h^0(r)$ is the relaxed radial profile, while $u_r(r,\theta)$ and $u_\theta(r,\theta)$ are correctors in the radial and angular directions, respectively, whose precise choice will be specified later. The out-of-plane component $w(r,\theta)$ carries oscillations with amplitude $A(r)$ and is modulated by a smooth mask $m(\theta)$, which localizes the wrinkling frequencies near the optimal one. This ansatz is designed so that the excess length $\gamma(r)$ matches the target $\gamma_0(r)$ dictated by $\u_h^0$. A discrete-to-continuous comparison ensures that the discrepancy $\gamma-\gamma_0$ is exponentially small.
	
	\item \textbf{Definition of $w$ and control of $W_r-W_\rel$.}  
	We verify that the mismatch term $W_r-W_\rel$ remains small by choosing frequencies close to the optimal value. This step establishes the first control on the oscillatory contribution of $w$.
	
	\item \textbf{Estimates for $B^2$ and $\sigma_h^0B$.}  
	We compute the bending term $B$ and the mixed term $\sigma_h^0B$, showing that they are bounded in terms of $A(r)$ and its derivatives. These estimates ensure that the bending cost is consistent with the targeted scaling.
	
	\item \textbf{Expansion of the remainder term $R_h$ and estimates for the $w$-dependent contributions.}  
	The remainder term $R_h$ decomposes into five parts, $R_h^1,\dots,R_h^5$, each arising from distinct oscillatory interactions. Among these, $R_h^4$ and $R_h^5$ depend explicitly on $w$ and are estimated at this stage, while the dominant components in $R_h^1$, $R_h^2$, and $R_h^3$ are left for cancellation in the subsequent steps.
	
	\item \textbf{Definition of $u_\theta$ and reduction of $R_h^2$.}  
	We introduce the angular corrector $u_\theta$, chosen so that its derivatives cancel the oscillatory part of $R_h^2$, thereby reducing the remainder to lower order.
	
	\item \textbf{Definition of $u_r$, cancellation of $R_h^3$, and reduction of $R_h^1$.}  
	We construct the radial corrector $u_r=u_{r,1}+u_{r,2}$. The first part cancels the leading contribution in $R_h^3$, while the second part addresses a mixed term in $R_h^1$ that couples radial and angular derivatives. Together, these eliminate the dominant radial oscillations.
	
	\item \textbf{Estimates for $\partial_r u_{r,2}$ and final bound for $R_h$.}  
	The derivative $\partial_r u_{r,2}$ generates the remaining potentially large term in the expansion of $R_h$. We estimate it carefully and, together with the previous cancellations, obtain the final bound for $R_h$. At this stage, all five contributions $R_h^1,\dots,R_h^5$ are controlled.
	
	\item \textbf{Choice of the parameters $\alpha,\ell$, and $\de$.}  
	In the final step, the parameters are tuned according to the regime of $\beta$, so that bending, mismatch, and remainder are balanced. With this choice, the ansatz $(u,w)$ achieves the desired upper bound \eqref{toproveupperbound} in both regimes of $\beta$.
\end{enumerate}

\subsection{Proof of the upper bound part of the main results}

\begin{proof}[Proof of the upper bound in Theorem \ref{thm1}] The proof is divided in several steps.
	
	\begin{enumerate}[label=\textsc{\bf Step \arabic*.},ref=\textsc{\arabic*},leftmargin=0pt,labelsep=*,itemindent=*,itemsep=10pt,topsep=10pt]
		
	\item \textbf{The ansatz.}
	We consider a configuration $(u,w)$ of the form
	$$
	(u,w)=\Big(\bigl(\u_h^0(r)+u_r(r,\theta),\,u_\theta(r,\theta)\bigr),\,w(r,\theta)\Big),
	$$
	subject to the normalization
	$$
	\bar u_r(r)=\bar u_\theta(r)=\bar w(r)=0 \quad \forall r\in (0,r_0).
	$$
	Here $u_r$, $u_\theta$, and $w$ represent the oscillatory corrections: 
	the wrinkling profile $w$ is defined in~\eqref{eq:defw}, 
	the angular corrector $u_\theta$ in~\eqref{eq:defutheta}, 
	and the radial corrector $u_r$ in~\eqref{eq:defur}. 
	
	For such a configuration, the energy decomposes as
	\begin{multline}\label{difEnergies}
		E_h(u,w)-F_h^0(\u_h^0)=\int_0^{r_0} \bigg[2\sigma_h^0B+B^2+W_r\left(\frac{\u_h^0}r,w\right) -W_\rel\left(\frac{\u_h^0}r\right) \bigg]r\dr+R_h(u,\xi)+h^2\frac{r_0^2}{2R^2},
	\end{multline}
	where
	$$
	B(r)=\fint_0^{2\pi} |\partial_r w(r,\theta)|^2 \,\mathrm{d}\theta, 
	\qquad \sigma_h^0(r)=\u_h^0\,'(r)+\frac{r^2}{2R}, 
	\qquad \xi(r,\theta)=w(r,\theta)-\frac{r^2}{2R^2}.
	$$
	
	\item 
	\textbf{Definition of $w$ and control of $W_r-W_\rel$.}
	Our aim in this step is to define the wrinkling profile $w$ so that the excess length 
	$\gamma(r)$ matches the target $\gamma_0(r)$ dictated by $\u_h^0$. 
	This will ensure that the contribution $W_r-W_\rel$ in the energy remains small.
	
	\begin{enumerate}[label=\textsc{\bf Step \arabic{enumi}.\arabic*.},ref=\textsc{\arabic*},leftmargin=0pt,labelsep=*,itemindent=*,itemsep=10pt,topsep=10pt]
		
	\item \textbf{Construction of the wrinkling profile.}
	We introduce two auxiliary ingredients. First, a smooth bump function $m:\R\to[0,1]$, supported in $[-1/2,1/2]$, defined by
	\begin{equation*}
		m(t)\colonequals 
		\left\{
		\begin{array}{cl}
			\exp\biggl(-\dfrac1{1-4|t|^2}\biggr) &\mathrm{if}\ |t|\leq \frac12,\\
			0&\mathrm{if}\ |t|> \frac12
		\end{array}\right.
	\end{equation*}
	which localizes the wrinkling frequencies near the optimal one. Second, an amplitude function $A(r)$ (precisely defined in \eqref{eq:defA} below), supported in $[r_h,r_0]$, chosen later so that 
	$\gamma(r)\approx\gamma_0(r)$.
	
	\medskip
	With these definitions in place, and given parameters $\ell,\delta>0$ with 
	$N\colonequals h^\delta/\ell$, we set
	\begin{equation}\label{eq:defw}
		w(r,\theta)\colonequals A(r)rh^\frac\de2 \sum_{k>0} \mask \frac{\sqrt 2 \cos(kN\theta)}{kN},
	\end{equation}
	
	\medskip
	By construction $A(r)=0$ for $r\leq r_h$, so in the tensile region
	$$
	W_r\!\left(\frac{\u_h^0(r)}r,w(r,\cdot)\right)-W_\rel\!\left(\frac{\u_h^0(r)}r\right)=0.
	$$
	For $r\in(r_h,r_0]$, where $\u_h^0(r)< -2\para r$, the expansion \eqref{difW} yields
	\begin{multline}\label{difWWrel}
		W_r\left(\frac{\u_h^0}r,w(r,\cdot)\right) -W_\rel\left(\frac{\u_h^0}r\right)=\left|\gamma (r)-\gamma_0(r)\right|^2\\+A^2(r)h^\de \sum_k \maskk\left(\frac{hNk}r-\frac{\alpha_s^\frac12 r}{h^\frac\beta2 Nk} \right)^2
	\end{multline}
	with 
	$$
	\gamma(r)\colonequals A^2(r)h^\de \sum_k \maskk,\qquad \gamma_0(r)\colonequals {-}\left(\frac{ \u_h^0}r +2\para  \right) > 0.
	$$
	
	\item \textbf{Control of the excess length $\gamma(r)-\gamma_0(r)$.}
	To control $\gamma(r)-\gamma_0(r)$ we introduce the cutoff $\eta$ with $\eta(t)=1$ for $t>2$ and $\eta(t)=0$ for $t<1$, and set
	\begin{equation}\label{eq:defA}
		A(r)\colonequals 
		\left\{
		\begin{array}{cl}
			0&r\in (0,r_h)\\
			\eta\left(\dfrac{r-r_h}{h^\alpha}\right) \left(\dfrac{\gamma_0(r)}{\int_\R m^2 \dx}\right)^\frac12 &r\in (r_h,r_0),
		\end{array}\right.
	\end{equation}
	with $\alpha>0$ chosen later in such a way that the relation $h^\alpha \ll r_h$ holds. This guarantees that $A(r)$ grows smoothly from $0$ near $r_h$ to the desired amplitude.
	
	\medskip
	A direct computation shows
	\begin{equation}\label{derA}
		|A'(r)|\leq C\left(h^{-\alpha} r_h\right)^\frac12,\qquad |A''(r)|\leq C\left(h^{-3\alpha}r_h\right)^\frac12.
	\end{equation}
	
	\medskip
	To compare the discrete sum $\gamma(r)$ with the target $\gamma_0(r)$, 
	we define
	$$
	\tilde \gamma(r)\colonequals A^2(r)h^\de \int_\R \maskk \mathrm{d}k=A^2(r)\int_\R m^2 \dx.
	$$
	Then $\tilde\gamma(r)=\gamma_0(r)$ for $r\in(r_h+2h^\alpha,r_0)$, while in the transition region
	$$
	|\tilde\gamma(r)-\gamma_0(r)|\leq C h^\alpha r_h,\quad r\in[r_h,r_h+2h^\alpha],
	$$
	which implies
	\begin{equation}\label{gammas1}
		\int_{r_h}^{r_0}|\tilde\gamma(r)-\gamma_0(r)|^2r\dr\leq C(h^\alpha r_h)^3.
	\end{equation}
	
	\medskip
	Finally, to estimate $\gamma(r)-\tilde\gamma(r)$ we use the discrete-to-continuous comparison 
	from~\cite{BeKo17}, Eq.~(5.4). For smooth compactly supported $f:\R\to\R$, there exists $C_f>0$ 
	(support-dependent) such that, for any $0\neq n\in\N$, $t\in(0,1)$, and $\zeta\in\R$,
	\begin{equation}\label{difsumint}
		\biggl|t\sum_{k\in\Z}f(tk+\zeta)-\int_\R f\biggr|\leq C_ft^n\|f^{(n)}\|_\infty.
	\end{equation}
	Moreover, for any $n\in\N$,
	$$
	\|(m^2)^{(n)}\|_\infty\leq(C_mn)^{(C_m'n)},
	$$
	with constants $C_m,C_m'$ depending only on $m$, which may change from line to line. Combining with \eqref{difsumint} for $f=m^2$, we obtain
	$$
	|\gamma(r)-\tilde\gamma(r)|\leq(C_mn)^{(C_m'n)}(h^\de)^n,\quad\forall n\in\N.
	$$
	The crucial point is the choice of $n$: by optimizing we take
	\begin{equation*}
		n=(C_me^\frac1{C_m'})^{-1}h^{-\frac \de{C_m'}},
	\end{equation*}
	which yields
	$$
	\int_{r_h}^{r_0}|\gamma(r)-\tilde \gamma(r)|^2r\dr\leq C \exp\left(-\frac{h^{-\frac{\de}{C_m'}}}{C_m}\right)^2.
	$$
	This optimized choice of $n$ is a main novelty of our construction: it improves the error term in the upper bound in~\cite{BeKo17}, and thus yields a sharper result when $\beta=2$. Combined with \eqref{gammas1}, we obtain
	\begin{equation}\label{gammas2}
		\int_{r_h}^{r_0} |\gamma(r)-\gamma_0(r)|^2r\dr\leq C\left( (h^\alpha r_h)^3 +\exp\left(-\frac{h^{-\frac{\de}{C_m'}}}{C_m}\right)^2\right).
	\end{equation}
	
	\item \textbf{Frequency mismatch estimate and final bound.}
	Having controlled the discrepancy $\gamma(r)-\gamma_0(r)$ up to small errors, we now turn to the remaining contribution in \eqref{difWWrel}. 
	This second term reflects the frequency mismatch in the wrinkling profile, 
	and its estimation completes the control of $W_r-W_\rel$ within this step.
	
	\medskip
	Define
	$$
	g(k)\colonequals \frac{hNk}r-\frac{\alpha^{\frac12} r}{h^{\frac\beta2} Nk}.
	$$
	Observe that
	$g(k_0)=0$, where $k_0\colonequals \frac{\alpha_s^\frac14 r}{\factor N}$ is the frequency $k$ for which 
	\begin{equation*}
		\mask =m(0)=1.
	\end{equation*} 
	Moreover, $g'(k_0)=\frac{2hN}r$. 
	
	The sum $\sum\limits_k \mask^2$ then runs over $O(h^{-\delta})$ integers $k$ with $|k-k_0|\leq \frac12 h^{-\delta}$, a fact that will be invoked repeatedly in the sequel.
	A Taylor expansion and the definition of $N$ thus give
	$$
	h^\delta \sum_k \mask^2 g(k)^2
	\leq C h^\delta h^{-\delta}\bigl(g'(k_0) h^{-\delta}\bigr)^2
	\leq \frac{C}{r^2}\frac{h^2}{\ell^2}.
	$$
	Using \eqref{v0}, we deduce that $\frac{A^2(r)}{r^2}$ is bounded, and therefore
	$$
	\int_0^{r_0} A^2(r)h^\de \sum_k \maskk\left(\frac{hNk}r-\frac{\alpha_s^\frac12 r}{h^\frac\beta2 Nk} \right)^2r\dr\leq C\frac{h^2}{\ell^2}.
	$$
	Together with \eqref{gammas2}, in view of \eqref{difWWrel}, this yields
	\begin{equation}\label{upper1}
		\int_0^{r_0} \left[W_r\left(\frac{\u_h^0}r,w\right) -W_\rel\left(\frac{\u_h^0}r\right) \right]r\dr\leq C\left((h^\alpha r_h)^3+\exp\left(-h^{-\frac{\de}{C_m'}}\right)^2+\frac{h^2}{\ell^2}\right).
	\end{equation}
	\end{enumerate}
	
	\item \textbf{Estimates for $B^2$ and $\sigma_h^0B$.}  
	Having fixed the wrinkling profile $w$, we now estimate the contributions $\sigma_h^0 B$ and $B^2$ that enter the energy difference in \eqref{difEnergies}. 
	The bounds follow directly from the properties of $A(r)$ and its derivatives.
		
	We have
	\begin{equation*}
		B(r)=\fint_0^{2\pi}|\partial_r w(r,\theta)|^2\dt=h^\de \sum_k \left(\partial_r\left( A(r)r\mask\right) \right)^2\frac1{(kN)^2}.
	\end{equation*}
	Over the $O(h^{-\delta})$ frequencies with $|k-k_0|\leq \frac12 h^{-\delta}$,
	\begin{equation}\label{kN}
		kN\geq C\frac{r}{\factor}.
	\end{equation}
	From \eqref{derA} together with $\frac{A^2(r)}{r^2}\leq C$, we thus obtain
	$$
	B(r)\leq C \factors \left(|A'(r)|^2+1+\frac{\ell^2}{\factors}\right)
	\leq C(\factors h^{-\alpha}r_h+\ell^2).
	$$
	A direct computation shows that $0\leq \sigma_h^0 r\leq Ch^{\frac{2-\beta}{2}}$ for $r\geq r_h$ 
	(recall \eqref{sigma0}); thus
	\begin{equation}\label{sigmaB}
		\int_0^{r_0}\sigma_h^0 B(r)\, r\dr \leq C\bigl(h^{2-\alpha}r_h+h^{\frac{2-\beta}{2}}\ell^2\bigr)
	\end{equation}
	and
	\begin{equation}\label{B2}
		\int_0^{r_0} B^2(r)\, r\dr \leq C\bigl(h^{2+\beta-2\alpha}r_h^2+\ell^4\bigr).
	\end{equation}
	
	\item \textbf{Expansion of the remainder term $R_h$ and estimates for the $w$-dependent contributions.}
	We expand the remainder term $R_h$ into five contributions $R_h^1,\dots,R_h^5$, each isolating a distinct oscillatory or bending effect. 
	The purpose of this decomposition is to separate the parts that can be controlled directly from those requiring the introduction of the correctors $u_r,u_\theta$. 
	In particular, the terms $R_h^4$ and $R_h^5$ depend explicitly on the wrinkling profile $w$ and can be estimated immediately using the frequency bounds and the regularity of $A(r)$. 
	
	We proceed to introduce the five components $R_h^1,\dots,R_h^5$:	
	\begin{align}
	R_h^1&= \int_0^{r_0}\fint_0^{2\pi} \left|\partial_r u_r+\frac{(\partial_r\xi)^2}2-\frac{\overline{(\partial_r\xi)^2}}2 \right|^2\dt r\dr\label{def:Rh1}\\
	R_h^2&= \int_0^{r_0}\fint_0^{2\pi} \left|\frac{\partial_\theta u_\theta}r +\frac{u_r}r+\frac{(\partial_\theta w)^2}{2r^2}-\frac{\overline{(\partial_\theta w)^2}}{2r^2} \right|^2\dt r\dr\label{def:Rh2}\\
	R_h^3&= \int_0^{r_0}\fint_0^{2\pi} \frac12 \left|\frac{\partial_\theta u_r}r+r\partial_r\left(\frac{u_\theta}r\right)+\frac1r \partial_r\xi\partial_\theta\xi \right|^2\dt r\dr\label{def:Rh3}\\
	R_h^4&= \int_0^{r_0}\fint_0^{2\pi} h^2\left|\partial_{rr} w \right|^2\dt r\dr\label{def:Rh4}\\
	R_h^5&= \int_0^{r_0}\fint_0^{2\pi} \frac{2h^2}{r^2}\left|\partial_{\theta r}\xi \right|^2\dt r\dr\label{def:Rh5},
	\end{align}
	where $\overline{(\partial_r\xi)^2}=\fint_0^{2\pi}(\partial_r\xi)^2\mathrm{d}\theta$ and $\overline{(\partial_\theta w)^2}=\fint_0^{2\pi}(\partial_\theta w)^2\mathrm{d}\theta$.

	\begin{enumerate}[label=\textsc{\bf Step \arabic{enumi}.\arabic*.},ref=\textsc{\arabic*},leftmargin=0pt,labelsep=*,itemindent=*,itemsep=10pt,topsep=10pt]
	\item \textbf{Estimating $R_h^4$.} Since
	$$
	\partial_{rr} w= h^\frac\de2 \sum_k \partial_{rr}\left( A(r)r\mask\right) \frac{\sqrt 2\cos(kN\theta)}{kN},
	$$
	using \eqref{kN} and $\frac{|A'(r)|}{r}\leq C|A''(r)|$, we obtain
	$$
	|\partial_{rr}w|\leq Ch^{\delta/2} h^{-\delta}\,\factor \left(|A''(r)|+\frac{\ell^2}{\factor}\right),
	$$
	which combined with \eqref{derA} yields (recall \eqref{def:Rh4})
	\begin{equation}\label{Rh4}
		R_h^4\leq C h^{-\delta}\left(\factor\, h^{2-3\alpha}r_h+\ell^4 h^{\frac{2-\beta}{2}}\right).
	\end{equation}
	
	\item \textbf{Estimating $R_h^5$.} We have
	$$
	\partial_{\theta r} \xi = -h^\frac\de2 \sum_k \partial_r\left( A(r)r\mask\right) \sqrt 2 \sin(kN\theta),
	$$
	hence
	$$
	\frac{|\partial_{\theta r} \xi|}r \leq C h^\frac\de2 h^{-\de} \left(|A'(r)|+\frac\ell\factor \right).
	$$
	This combined with \eqref{derA} implies (recall \eqref{def:Rh5})
	\begin{equation}\label{Rh5}
		R_h^5\leq C h^{-\delta}\left(h^{2-\alpha}r_h+\ell^2 h^{\frac{2-\beta}{2}}\right).
	\end{equation}
	\end{enumerate}
	
	This yields the desired control of the $w$-dependent terms $R_h^4$ and $R_h^5$. 
	In the following steps we turn to the construction of $u_\theta,u_r$ and the control of $R_h^1,R_h^2,R_h^3$.
	
	\item \textbf{Definition of $u_\theta$ and reduction of $R_h^2$.}
	We now turn to $R_h^2$. 
	This term involves oscillations of $(\partial_\theta w)^2$ around its mean, which cannot be estimated directly. 
	To cancel these oscillations, we introduce an auxiliary angular displacement $u_\theta$, defined so that its derivative exactly compensates the fluctuating part of $(\partial_\theta w)^2$. 
	This construction reduces $R_h^2$ to a simpler form involving only the radial displacement $u_r$.
	
	To this end, we analyze $R_h^2$, defined in \eqref{def:Rh2}, in detail. We begin by observing that
	$$
	(\partial_\theta w)^2=2A^2(r)r^2h^\de \sum_{k,j}\mask \maskj \sin(kN\theta)\sin(jN\theta)
	$$
	and
	$$
	\overline{(\partial_\theta w)^2}= A^2(r)r^2h^\de\sum_k \maskk.
	$$
	Writing
	$$
	\sin(kN\theta)\sin(jN\theta)=-\frac12\cos((k+j)N\theta)+ \frac12\cos((k-j)N\theta),
	$$
	we obtain
	\begin{align*}
		(\partial_\theta w)^2-\overline{(\partial_\theta w)^2}
		=&- A^2(r)r^2h^\de \sum_{k,j}\mask \maskj \cos((k+j)N\theta) \\
		&+ A^2(r)r^2h^\de \sum_{k\neq j} \mask \maskj \cos((k-j)N\theta).
	\end{align*}
	Since these terms are of order $O(h^{-\de})$, we define $u_\theta$ so that they cancel. More precisely, we let
	\begin{equation}\label{eq:defutheta}
	u_\theta(r,\theta)= \frac12 (u_{\theta,+}-u_{\theta,-}),
	\end{equation}
	where
	\begin{align*}
		u_{\theta,+}(r,\theta)&\colonequals A^2(r)rh^\de \sum_{k,j}\mask \maskj \frac{\sin((k+j)N\theta)}{(k+j)N},\\
		u_{\theta,-}(r,\theta)&\colonequals A^2(r)rh^\de \sum_{k\neq j}\mask \maskj \frac{\sin((k-j)N\theta)}{(k-j)N}.
	\end{align*}
	Note that
	$$
	\frac{\partial_\theta u_\theta}r+\frac{(\partial_\theta w)^2}{2r^2}-\frac{\overline{(\partial_\theta w)^2}}{2r^2}=0,
	$$
	and therefore (recall \eqref{def:Rh2})
	\begin{equation}\label{R2}
		R_h^2=\int_0^{r_0}\fint_0^{2\pi}\left|\frac{u_r}r\right|^2\dt r\dr.
	\end{equation}
	
	This completes the construction of $u_\theta$ and the reduction of $R_h^2$ to a form depending solely on $u_r$. 
	In the following step, we introduce $u_r$, which allows us to control $R_h^1$ and eliminate $R_h^3$.
	
	\item \textbf{Definition of $u_r$, cancellation of $R_h^3$, and reduction of $R_h^1$.}
	We now construct the radial displacement $u_r$. 
	This function serves two essential purposes: it cancels the mixed oscillatory term in $R_h^3$, 
	and it reduces $R_h^1$ to a form that can be estimated effectively. 
	To achieve this, we split $u_r$ into two components, $u_{r,1}$ and $u_{r,2}$, each designed for a specific cancellation.
	
	\begin{enumerate}[label=\textsc{\bf Step \arabic{enumi}.\arabic*.},ref=\textsc{\arabic*},leftmargin=0pt,labelsep=*,itemindent=*,itemsep=10pt,topsep=10pt]
	\item \textbf{Analysis of $R_h^1$.} We now turn to $R_h^1$, defined in \eqref{def:Rh1}. Observe that
	$$
	(\partial_r \xi)^2=\left(\partial_r w-\frac rR \right)^2=(\partial_r w)^2-2\frac rR \partial_r w +\frac{r^2}{R^2},
	$$
	and since $\overline{\partial_r w}=0$, we obtain 
	\begin{equation}\label{xi2}
		\frac12\left((\partial_r \xi)^2-\overline{(\partial_r \xi)^2}\right)=\frac12 \left((\partial_r w)^2-B\right)-\frac rR \partial_r w.
	\end{equation}
	We begin by analyzing the first term on the right-hand side of \eqref{xi2}. Introducing
	\begin{equation}\label{defMk}
		M_k(r)\colonequals A(r)r\mask,
	\end{equation}
	we have
	$$
	(\partial_r w)^2=2h^\de \sum_{k,j} \partial_r M_k(r) \partial_r  M_j(r)\frac{\cos(kN\theta)\cos(jN\theta)}{kjN^2}.
	$$
	Arguing as when estimating $B$, we deduce
	$$
	|(\partial_r w)^2|\leq C h^\de h^{-2\de}\factors\left(|A'|^2+\frac{\ell^2}{\factors}\right)\leq Ch^{-\de}(\factors h^{-\alpha}r_h+\ell^2).
	$$
	Consequently, recalling \eqref{B2}, we obtain
	\begin{equation}\label{Rh1}
		\int_0^{r_0}\fint_0^{2\pi}\left|\frac12\left((\partial_r w)^2-B\right)\right|^2\dt r\dr\leq Ch^{-2\de}(h^{2+\beta-2\alpha}r_h^2+\ell^4).
	\end{equation}
	
	By contrast, the second term in the right-hand side of \eqref{xi2} is of order too large to be controlled directly. 
	
	\item \textbf{Definition of $u_r$ and construction of $u_{r,1}$.}
	To address this issue, we define
	\begin{equation}\label{eq:defur}
	u_r(r,\theta)\colonequals u_{r,1}(r,\theta)+u_{r,2}(r,\theta)\quad \mathrm{with}\ \bar u_{r,1}(r)=\bar u_{r,2}(r)=0\quad \forall r\in[0,r_0],
	\end{equation}
	where   
	$$
	u_{r,1}(r,\theta) \colonequals \frac rR w(r,\theta),
	$$ 
	and $u_{r,2}$ is defined in \eqref{def:ur2} below. Note that 
	$$
	\partial_r u_{r,1}-\frac rR \partial_r w=\frac wR.
	$$
	This choice of $u_{r,1}$ is crucial, since it cancels the $O(1)$ contribution $\frac{r}{R}\partial_\theta w$ appearing in $R_h^3$ (see \eqref{termR3} below).
	
	Moreover, using \eqref{eq:defw}, we deduce that
	$$
	\int_0^{r_0}\fint_0^{2\pi}\left|\partial_r u_{r,1}-\frac rR \partial_r w\right|^2\dt r \dr=\int_0^{r_0}\fint_0^{2\pi}\left|\frac wR\right|^2\dt r \dr\leq C h^{-\de}\factors,
	$$
	which, combined with \eqref{Rh1}, yields (recall \eqref{def:Rh1})
	\begin{equation}\label{R1}
		R_h^1\leq \int_0^{r_0}\fint_0^{2\pi}\left|\partial_r u_{r,2}\right|^2\dt r\dr + Ch^{-2\de}\left(h^{2+\beta-2\alpha}r_h^2+\ell^4+\factors\right).
	\end{equation}
	
	\item \textbf{Cancellation of $R_h^3$ via $u_{r,2}$.}
	We now turn to $R_h^3$, defined in \eqref{def:Rh3}. First, we observe that 
	\begin{equation}\label{termR3}
		\partial_\theta u_{r,1}+\partial_r\xi\partial_\theta \xi=\frac rR \partial_\theta w+ \partial_r w \partial_\theta w-\frac rR \partial_\theta w=\partial_r w \partial_\theta w.
	\end{equation}
	This term and $\partial_r \left(\frac{u_\theta}r\right)$ are both too large. 
	For the latter, the largeness stems from the fact that differentiating $u_\theta$ in $r$ cancels the good factors $\frac{1}{(k\pm j)N}$ present in \eqref{eq:defutheta}.
	To remedy this, we design $u_{r,2}$ so that
	\begin{equation}\label{equationforur2}
	\frac{\partial_\theta u_{r,2}}r+r\partial_r\left(\frac{u_\theta}r\right)+\frac1r\partial_r w \partial_\theta w=0,
		\end{equation}
	which immediately implies $R_h^3=0$.
	
	\item \textbf{Construction of $u_{r,2}$.}
	We proceed in two steps. First, set $U_r(r,\theta)=\frac12 (U_{r,+}-U_{r,-})$, where
	\begin{align*}
		U_{r,+}(r,\theta)&\colonequals r^2h^\de \sum_{k,j}\partial_r\left( A^2(r)\mask \maskj\right) \frac{\cos((k+j)N\theta)}{(k+j)^2N^2},\\
		U_{r,-}(r,\theta)&\colonequals r^2h^\de \sum_{k\neq j}\partial_r\left( A^2(r)\mask \maskj\right) \frac{\cos((k-j)N\theta)}{(k-j)^2N^2},
	\end{align*}
	and observe that
	$$
	\frac{\partial_\theta U_r}r+r\partial_r \left(\frac{u_\theta}r\right)=0.
	$$
	Next, using the identity
	$$
	2\cos(kN\theta)\sin(jN\theta)=\sin((k+j)N\theta)-\sin((k-j)N\theta),
	$$
	and recalling \eqref{defMk}, we obtain
	\begin{align*}
		\partial_r w\partial_\theta w(r,\theta)=&-h^\de \sum_{k,j}(\partial_rM_k(r)) M_j(r) \frac{\sin((k+j)N\theta)}{kN}\\
		&+h^\de \sum_{k\neq j}(\partial_rM_k(r)) M_j(r) \frac{\sin((k-j)N\theta)}{kN}.
	\end{align*}
	Finally, define $V_r(r,\theta)=V_{r,+}-V_{r,-}$, where
	\begin{align*}
		V_{r,+}(r,\theta)&\colonequals -h^\de \sum_{k,j}(\partial_rM_k(r)) M_j(r) \frac{\cos((k+j)N\theta)}{k(k+j)N^2},\\
		V_{r,-}(r,\theta)&\colonequals -h^\de \sum_{k\neq j}(\partial_rM_k(r)) M_j(r) \frac{\cos((k-j)N\theta)}{k(k-j)N^2},
	\end{align*}
	so that
	$$
	\partial_\theta V_r+\partial_r w\partial_\theta w=0.
	$$
	We finally set 
	\begin{equation}\label{def:ur2}
		u_{r,2}\colonequals U_r+V_r,
	\end{equation}
	which ensures that \eqref{equationforur2} holds.
	\end{enumerate}
	
	This completes the construction of $u_r$. 
	The first component $u_{r,1}$ cancels the $O(1)$ term in $R_h^3$, while the second component $u_{r,2}$ eliminates the remaining oscillatory contributions. 
	As a result, $R_h^3$ vanishes entirely and $R_h^1$ reduces to a form involving only $\partial_r u_{r,2}$. 
	In the next step we estimate $\partial_r u_{r,2}$ and combine all contributions to obtain the final bound for $R_h$.

	\item \textbf{Estimates for $\partial_r u_{r,2}$ and final bound for $R_h$.} 
	
	In this step we establish the quantitative bounds for the radial derivative of $u_{r,2}$, which is the decisive term in the control of $R_h$.  
	Recalling \eqref{R2} and observing that $|\partial_r u_{r,1}|$ is of lower order compared to $|\partial_r u_{r,2}|$, we have
	$$
	R_h^2\leq C\int_0^{r_0}\fint_0^{2\pi}\left|\partial_r u_{r,2}\right|^2\dt r\dr,
	$$
	which, combined with \eqref{R1} and the fact that $R_h^3=0$, implies
	\begin{equation}\label{R123}
		R_h^1+R_h^2+R_h^3\leq C\int_0^{r_0}\fint_0^{2\pi}\left|\partial_r u_{r,2}\right|^2\dt r\dr + Ch^{-2\de}(h^{2+\beta-2\alpha}r_h^2+\ell^4+\factors).
	\end{equation}
	In addition, from the definition of $u_{r,2}$ we deduce that
	\begin{equation}\label{step7eq1}
	\int_0^{r_0}\fint_0^{2\pi}\left|\partial_r u_{r,2}\right|^2\dt r\dr\leq  
	\int_0^{r_0}\fint_0^{2\pi}\left|\partial_r (U_{r,-}+V_{r,-})\right|^2\dt r\dr,
	\end{equation}
	that is, the terms on the right-hand side are the leading-order contributions to $|\partial_r u_{r,2}|$.
	
	\begin{enumerate}[label=\textsc{\bf Step \arabic{enumi}.\arabic*.},ref=\textsc{\arabic*},leftmargin=0pt,labelsep=*,itemindent=*,itemsep=10pt,topsep=10pt]
	\item \textbf{Estimate for $\partial_r U_{r,-}$.}  
	We begin with $\int_0^{r_0}\fint_0^{2\pi}\left|\partial_r U_{r,-}\right|^2\dt r\dr$.  
	We define
	$$
	H(r)\colonequals h^\de \sum_{k\neq j}\mask \maskj\frac{\cos((k-j)N\theta)}{(k-j)^2N^2}.
	$$
	By definition, we have
	$$
	\partial_r U_{r,-}(r,\theta) =\partial_r (r^2\partial_ r(A^2(r)))H(r)+2[r^2\partial_r (A^2(r))+rA^2(r)]\partial_r H(r)+r^2A^2(r)\partial_{rr}H(r).
	$$
	A straightforward computation shows that
	$$
	|r^2\partial_r (A^2(r))+rA^2(r)|+|r^2A^2(r)|\leq C
	$$
	and
	$$
	|\partial_r (r^2\partial_ r(A^2(r)))|\leq C(1+h^{-\alpha}r_h^3)\leq C(1+h^\frac{2-\beta}2h^{-\alpha}).
	$$
	Thus
	\begin{equation}\label{h0}
		|\partial_r U_{r,-}(r,\theta)|\leq C(1+h^\frac{2-\beta}2h^{-\alpha})|H(r)|+C\left(|\partial_r H(r)| +|\partial_{rr} H(r)|\right).
	\end{equation}
	Next, we note that  
	\begin{equation}\label{h1}
		|H(r)|\leq \frac{Ch^{-\de}}{N^2}.
	\end{equation}
	In addition, we observe that $H(r)$ is $\frac\factor{\alpha_s^\frac14 \ell} h^\de$-periodic. Then, a direct computation shows that, for any $n\in \N$,  
	$$
	\|m^{(n)}\|_\infty\leq (C_m n)^{(C_m' n)},
	$$
	which allows us to deduce that
	$$
	\|H^{(n)}\|_\infty \leq (C_m n)^{(C_m' n)} \frac{h^{-\de}}{N^2}\left(\frac{\alpha_s^\frac14 \ell }\factor\right)^n.
	$$
	By observing that, for any $n\geq 2$, $H^{(n-1)}$ has mean zero, due to the periodicity of $H^{(n-2)}$, we deduce that there exists $t_{n-1}\in \left(0,\frac\factor{\alpha^\frac14 \ell} h^\de\right)$ such that $H^{(n-1)}(t_{n-1})=0$. Therefore, by Taylor expansion, for any $s\in \left(0,\frac\factor{\alpha_s^\frac14 \ell} h^\de\right)$, we have
	$$
	|H^{(n-1)}(s)|\leq \|H^{(n)}\|_\infty \frac\factor{\alpha^\frac14 \ell }  h^\de\leq  (C_m n)^{(C_m' n)} \frac{h^{-\de}}{N^2}\left(\frac{\alpha^\frac14 \ell }\factor\right)^n\frac\factor{\alpha_s^\frac14 \ell} h^\de= (C_m n)^{(C_m' n)} \frac{1}{N^2}\left(\frac{\alpha^\frac14 \ell }\factor\right)^{n-1}.
	$$
	We can now do the same to estimate $H^{(n-2)}$ in terms of our updated bound on the supreme norm of $H^{(n-1)}$, which yields that for any $s\in \left(0,\frac\factor{\alpha_s^\frac14 \ell} h^\de\right)$, we have
	\begin{multline*}
		|H^{(n-2)}(s)|\leq \|H^{(n-1)}\|_\infty \frac\factor{\alpha^\frac14 \ell }  h^\de\leq  (C_m n)^{(C_m' n)} \frac{1}{N^2}\left(\frac{\alpha^\frac14 \ell }\factor\right)^{n-1}\frac\factor{\alpha_s^\frac14 \ell} h^\de\\
		= (C_m n)^{(C_m' n)} \frac{1}{N^2}\left(\frac{\alpha^\frac14 \ell }\factor\right)^{n-2}h^{\delta}.
	\end{multline*}
	Iterating this procedure, we obtain
	\begin{align*}
		|H'|+|H''|&\leq (C_m n)^{(C_m' n)}\frac 1{N^2}\left( \frac{\alpha_s^\frac14 \ell }\factor\right)^2 h^{\de(n-2)}\\
		&\leq (C_m n)^{(C_m' n)} \left(\frac{\ell^2}\factor\right)^2 h^{\de(n-4)}.
	\end{align*}
	Optimizing over $n$ on the right-hand side, that is, choosing  
	$$
	n=(C_m e^\frac1{C_m'})^{-1}h^{-\frac \de{C_m'}}
	$$  
	yields
	$$
	|H'|+ |H''|\leq C\left(\frac{\ell^2}\factor\right)^2h^{-4\de}\exp\left(-\frac{h^{-\frac{\de}{C_m'}}}{C_m}\right).
	$$
	Plugging in this and \eqref{h1} into \eqref{h0}, yields  
    \begin{equation}\label{step7eq2}
	\int_0^{r_0}\fint_0^{2\pi}\left|\partial_r U_{r,-}\right|^2\dt r\dr\leq C(1+h^{2-\beta-2\alpha})\ell^4h^{-6
		\de}+ C\left(\frac{\ell^2}\factor\right)^4h^{-8\de}\exp\left(-\frac{h^{-\frac{\de}{C_m'}}}{C_m}\right)^2.
	\end{equation}
	
	\item \textbf{Estimate for $\partial_r V_{r,-}$.}  
	We now estimate $\int_0^{r_0}\fint_0^{2\pi}\left|\partial_r V_{r,-}\right|^2\dt r\dr$.  
	Arguing similarly as above, one deduces that
	$$
	|\partial_r V_{r,-}|\leq C \left((Ar)^2|\partial_{rr}H|+|\partial_r((Ar)^2)||\partial_r H|+|\partial_{rr}((Ar)^2)||H|\right).
	$$
	Using that
	$$
	|(Ar)^2|+|\partial_r((Ar)^2)|\leq C,\quad |\partial_{rr}((Ar)^2)|\leq C(1+h^\frac{2-\beta}2h^{-\alpha}),
	$$
	and the estimates on $H$ and its derivatives, we deduce that
    \begin{equation}\label{step7eq3}
	\int_0^{r_0}\fint_0^{2\pi}\left|\partial_r V_{r,-}\right|^2\dt r\dr\leq  C(1+h^{2-\beta-2\alpha})\ell^4h^{-6\de}+ C\left(\frac{\ell^2}\factor\right)^4h^{-8\de}\exp\left(-\frac{h^{-\frac{\de}{C_m'}}}{C_m}\right)^2.
	\end{equation}
	
	\item \textbf{Final bound.}  
	By combining \eqref{step7eq1}, \eqref{step7eq2}, and \eqref{step7eq3} with \eqref{R123}, we deduce that
	\begin{multline}\label{Rh123}
		R_h^1+R_h^2+R_h^3\leq C(1+h^{2-\beta-2\alpha})\ell^4h^{-6\de}+ C\left(\frac{\ell^2}\factor\right)^4h^{-8\de}\exp\left(-\frac{h^{-\frac{\de}{C_m'}}}{C_m}\right)^2 \\ +Ch^{-2\de}(h^{2+\beta-2\alpha}r_h^2+\ell^4+\factors).
	\end{multline}
	\end{enumerate}
	
	\item \textbf{Choice of the parameters $\alpha,\ell$, and $\de$.}
	We now reach the final stage of the proof, where the parameters $\alpha,\ell,\de$ are fixed according to the regime of $\beta$.  
	
	\begin{enumerate}[label=\textsc{\bf Case},leftmargin=0pt,labelsep=*,itemindent=*,itemsep=10pt,topsep=10pt]
	\item {$\bm{\frac23 \leq \beta \leq 2.}$}
	To fix $\ell$, we follow the same strategy as in the proof of the lower bound and match the terms depending on $\ell$ on the right-hand sides of \eqref{upper1} and \eqref{sigmaB}, that is
	$$
	\frac{h^2}{\ell^2}=h^\frac{2-\beta}2\ell^2,
	$$
	which gives $\ell=h^\frac{2+\beta}8$. In particular, this implies that
	$$
	\frac{h^2}{\ell^2}=h^\frac{2-\beta}2\ell^2=\expo.
	$$  
	Next, we choose $\alpha$ by considering the largest $\alpha$-dependent term in the estimate for $R_h$, namely $h^{2-\beta-2\alpha}\ell^4h^{-2\delta}$, and set $\alpha$ so that (up to a constant)
	$$
	h^{2-\beta-2\alpha}\ell^4=\expo,
	$$
	that is $\alpha= \frac{6-\beta}{8}$. This choice in particular ensures that $h^\alpha \ll r_h$, and moreover
	$$
	(h^\alpha r_h)^3+h^{2-\alpha}r_h+h^\frac{2+\beta}2h^{2-3\alpha}r_h+h^{2+\beta-2\alpha}r_h^2\leq C\expo.
	$$
	
	Therefore, we deduce from \eqref{upper1}, \eqref{sigmaB}, and \eqref{B2} that
	\begin{equation}\label{fin1}
		\int_0^{r_0} \Bigg[2\sigma_h^0B+B^2+W_r\left(\frac{\u_h^0}r,w\right) -W_\rel\left(\frac{\u_h^0}r\right) \Bigg]r\dr\leq C\left(\expo +\exp\left(-\frac{h^{-\frac{\de}{C_m'}}}{C_m}\right)^2\right).
	\end{equation}
	On the other hand, from \eqref{Rh4}, \eqref{Rh5}, and \eqref{Rh123}, we deduce that
	\begin{equation}\label{fin2}
		R_h\leq C\left(h^{-6\de}\expo +h^{-8\de}\exp\left(-\frac{h^{-\frac{\de}{C_m'}}}{C_m}\right)^2\right).
	\end{equation}
	Finally, we choose $\de$ in such a way that
	$$
	\expo=\exp\left(-\frac{h^{-\frac{\de}{C_m'}}}{C_m}\right)^2.
	$$
	This implies
	$$
	h^{-\de}=\left(C_m\log \frac1{h^\frac{6-\beta}8} \right)^{C_m'},
	$$
	that is
	$$
	\de=\dfrac{\log \left(C_m\log \frac1{h^\frac{6-\beta}8} \right)^{C_m'}}{\log \frac1h}.
	$$
	Hence, inserting this into \eqref{fin1} and \eqref{fin2}, and combining with \eqref{difEnergies}, we obtain
	$$
	E_h(u,w)-F_h^0(\u_h^0)\leq \left(C_m\log \frac1{h^\frac{6-\beta}8} \right)^{8C_m'}\expo.
	$$
	This completes the proof of the upper bound \eqref{toproveupperbound} in the regime $\frac23\leq \beta\leq 2$.
	
	\item {$\bm{0 <\beta < \frac23.}$}  
	Following once again the proof of the lower bound, we fix $\ell$ by matching the terms depending on $\ell$ on the right-hand sides of \eqref{upper1} and \eqref{B2}, that is
	$$
	\frac{h^2}{\ell^2}=\ell^4,
	$$
	which gives $\ell=h^\frac13$. In particular, this implies that
	$$
	\frac{h^2}{\ell^2}=\ell^4=\expos.
	$$
	We next determine $\alpha$ by considering the largest $\alpha$-dependent term in the estimate for $R_h$, namely $h^\frac{2+\beta}2h^{2-3\alpha}r_h$, and set $\alpha$ so that (up to a constant)
	$$
	h^\frac{2+\beta}2h^{2-3\alpha}r_h=\factors,
	$$
	that is $\alpha=\frac{14-\beta}{18}$. This choice ensures that $h^\alpha \ll r_h$, and moreover
	\begin{equation}\label{eqalpha}
		(h^\alpha r_h)^3+h^{2-\alpha}r_h+h^{2-\beta-2\alpha}\ell^4+h^{2+\beta-2\alpha}r_h^2\leq C\factors.
	\end{equation}
	
	We then deduce from \eqref{upper1}, \eqref{sigmaB}, and \eqref{B2} that
	\begin{equation*}
		\int_0^{r_0} \Bigg[2\sigma_h^0B+B^2+W_r\left(\frac{\u_h^0}r,w\right) -W_\rel\left(\frac{\u_h^0}r\right) \Bigg]r\dr\leq C\left(\factors +\exp\left(-\frac{h^{-\frac{\de}{C_m'}}}{C_m}\right)^2\right).
	\end{equation*}
	On the other hand, from \eqref{Rh4}, \eqref{Rh5}, and \eqref{Rh123}, we deduce that
	$$
	R_h\leq C\left(h^{-6\de}\factors +h^{-8\de}\exp\left(-\frac{h^{-\frac{\de}{C_m'}}}{C_m}\right)^2\right).
	$$
	Note that the term $\factors$, which for $\beta<\frac23$ dominates $h^\frac43$, essentially arises from the contribution $|w|^2$ in $R_h^1$. We were not able to eliminate it, and we believe that in the case $0\leq \beta<\frac23$ it is indeed the leading-order term in the expansion of $E_h(u,w)-F_h^0(\u_h^0)$. This fact actually drives the choice of $\alpha$ in \eqref{eqalpha}. In fact, if $\factors$ had not appeared in our energy estimates, we would have taken the right-hand side of \eqref{eqalpha} as $\expos$, analogously to the previous case.
	
	Since in our upper bound construction the only place where $\factors$ appears is in $R_h$, providing a matching lower bound would require extracting, for a minimizing configuration, a term of the same order from $R_h$, which we do not know how to do.  
	
	It is worth noticing that $\factors=h^\frac43$ for $\beta=\frac23$, and that $\factors \ll h^\frac{6-\beta}4$ if $\frac23<\beta\leq 2$, which is why this term is of lower order in the previously considered case.  
	
	\medskip
	Finally, we choose $\de$ in such a way that
	$$
	\factors=\exp\left(-\frac{h^{-\frac{\de}{C_m'}}}{C_m}\right)^2.
	$$
	This implies
	$$
	h^{-\de}=\left(C_m\log \frac1\factor \right)^{C_m'},
	$$
	that is
	$$
	\de=\dfrac{\log \left(C_m\log \frac1\factor \right)^{C_m'}}{\log \frac1h}.
	$$
	Hence, we deduce that
	$$
	E_h(u,w)-F_h^0(\u_h^0)\leq \left(C_m\log \frac1\factor \right)^{8C_m'}\factors.
	$$
	This completes the proof of the upper bound \eqref{toproveupperbound} in the regime $0<\beta< \frac23$.
	\end{enumerate}	
\end{enumerate}
\end{proof}

\section*{Data availability statement}
No datasets were generated or analyzed during the current study.
\section*{Conflict-of-interest statement}
The authors declare that they have no conflicts of interest.
\bibliographystyle{abbrv}
\bibliography{bella}

\end{document}